\pdfoutput=1 
\documentclass[english, reqno, a4paper,12pt]{amsart} 
\usepackage[utf8]{inputenc} 
\usepackage[T1]{fontenc} 
\usepackage{lmodern} 
\usepackage{amsmath}
\usepackage{amssymb} 
\usepackage{stmaryrd} 
\usepackage{mathtools} 
\usepackage{mathrsfs}  
\usepackage{graphicx} 
\usepackage{xcolor}   
\usepackage[english]{babel} 
\usepackage{thmtools} 
\usepackage{caption} 
\usepackage{subcaption}
\usepackage{hyperref} 
\usepackage{microtype} 
\usepackage[shortlabels]{enumitem} 
\usepackage{xargs} 
\usepackage[colorinlistoftodos,prependcaption,textsize=scriptsize]{todonotes} 
\usepackage{csquotes} 
\usepackage{dsfont} 
\usepackage[doi=false,isbn=false,backend=bibtex,style=alphabetic,maxbibnames=99,maxalphanames=99,giveninits=true]{biblatex} 
\setenumerate[0]{label = \textit{\arabic*.}, ref = \textit{\arabic*.}} 

\numberwithin{equation}{section} 
\mathtoolsset{showonlyrefs} 

\newtheorem{prop}{Proposition}[section]
\newtheorem{definition}[prop]{Definition}
\newtheorem{Lemma}[prop]{Lemma}

\newtheorem{example}[prop]{Example}
\newtheorem{theorem}[prop]{Theorem}
\newtheorem{remark}[prop]{Remark}
\newtheorem{corollary}[prop]{Corollary}

\newcommandx{\lucas}[2][1=]{\todo[linecolor=red,backgroundcolor=red!25,bordercolor=red,#1, author = Lucas]{#2}} 
\newcommandx{\lorenzo}[2][1=]{\todo[linecolor=blue,backgroundcolor=blue!25,bordercolor=blue,#1, author = Lorenzo]{#2}} 

\newcommand{\writefun}[5]{\ensuremath{\begin{array}[t]{lrcl}
#1 : & #2 & \longrightarrow & #3 \\
    & #4 & \longmapsto & #5 \end{array}}} 

\DeclareMathOperator\supp{supp} 

\newcommand{\R}{\mathbb{R}}

\newcommand{\cI}{\mathcal{I}}

\renewcommand{\d}{\, {\rm d}}

\title{The Sewing lemma for $0 < \gamma \leq 1$}
\date{\today}
\author{Lucas Broux and Lorenzo Zambotti}
\address[Lucas Broux]{Sorbonne Universit\'e, Laboratoire de
	Probabilit\'es, Statistique et Mod\'elisation, 4 Pl. Jussieu, 75005 Paris,France}
\email{\href{mailto:lucas.broux@upmc.fr}{\nolinkurl{lucas.broux@upmc.fr}}}
\address[Lorenzo Zambotti]{Sorbonne Universit\'e, Laboratoire de
	Probabilit\'es, Statistique et Mod\'elisation, 4 Pl. Jussieu, 75005 Paris,France}
\email{\href{mailto:zambotti@lpsm.paris}{\nolinkurl{zambotti@lpsm.paris}}}

\makeatletter
\@namedef{subjclassname@2020}{%
  \textup{2020} Mathematics Subject Classification}
\makeatother
\subjclass[2020]{60L20, 16T05}

\keywords{Rough Paths, Sewing Lemma, Hopf algebras, Regularity structures}

\begin{document}

\begin{abstract}
We establish a Sewing lemma in the regime $\gamma \in \left( 0, 1 \right]$, constructing a Sewing map which is neither unique nor canonical,
but which is nonetheless continuous with respect to the standard norms.
Two immediate corollaries follow, which hold on any commutative graded connected locally finite Hopf algebra: a simple constructive proof of the Lyons-Victoir extension theorem which associates to a H\"older path a rough path, with the additional result that this map can be made continuous; the bicontinuity of a transitive free action of a space of H\"older functions on the set of Rough Paths.
\end{abstract}

\maketitle

\tableofcontents

\section{Introduction}

Rough paths were introduced by T. Lyons in \cite{MR1654527} in order to give a robust theory for controlled
ordinary differential equations: for $Y:[0,T]\to\R^k$ of class $C^1$ and $\sigma:\R^d\to\R^d\otimes\R^k$ smooth,
one studies the equation
\[
X_t=X_0+\int_0^t\sigma(X_s)\, {\rm d}Y_s, \qquad t\geq 0,
\]
and the aim is to extend the map $Y\mapsto X$ to paths $Y$ of class $C^\alpha$ with $\alpha\in(0,1)$, in order
to include the case of Brownian motion and therefore stochastic differential equations.

A few years later, a new analytical tool was introduced to study such rough differential equations: the Sewing lemma \cite{MR2091358, MR2261056}, 
which allows to define uniquely a notion of integral $I_t = \int_{0}^t X_s \, {\rm d} Y_s$ in situations where $X$ and $Y$ may be paths of low regularity.
For example, in the so called Young regime, i.e.\ when $X$ and $Y$ have H\"older regularity $\alpha$ resp.\ $\beta$ with $\alpha + \beta > 1$, it is a classical result due to Young \cite{MR1555421} and Kondurar \cite{kondurar1937integrale} that a canonical integration theory
exists for $I_t = \int_{0}^t X_s \, {\rm d} Y_s$. The Sewing Lemma recovers this setting by showing that there exists \emph{one and only one} $I:[0,T]\to\R$ such that 
\[
I_0=0, \qquad \left| I_t - I_s - Y_s \left( X_t - X_s \right) \right| \lesssim \left| t - s \right|^{\alpha + \beta}.
\]

This raises the following general question: given $A \colon [0,T]^2 \to \mathbb{R}$ and $\gamma > 0$, does there exist $I \colon [0,T] \to \mathbb{R}$ satisfying $\left| I_t - I_s - A_{s, t} \right| \lesssim \left| t - s \right|^{\gamma}$ uniformly over $s \leq t \in [0,T]$?
The Sewing lemma gives a simple answer to this question when $\gamma > 1$: it asserts that such an $I$ exists 
(uniquely if $I_0=0$) as soon as $\left| A_{s, t} - A_{s, u} - A_{u, t} \right| \lesssim \left| t - s \right|^{\gamma}$ uniformly over $s \leq u \leq t \in [0,T]$.
This generalises Young-Kondurar's integration theory, and 
allows to build a well-posedness theory for rough differential equations.

Surprisingly, the Sewing lemma has not been extended yet to cover the case $\gamma\in(0,1]$. It is clear that
the situation is different, because this time the relation $\left| I_t - I_s - A_{s, t} \right| \lesssim \left| t - s \right|^{\gamma}$ does not characterise $I$ anymore: indeed, any other $\tilde I$ has the same property if 
and only if $I-\tilde I$ is $\gamma$-H\"older. However, existence of $I$ under the hypothesis that
$\left| A_{s, t} - A_{s, u} - A_{u, t} \right| \lesssim \left| t - s \right|^{\gamma}$ is not known. The aim of
this paper is to fill this gap, including the case $\gamma=1$ for which the result is slightly different.

As in the case of $\gamma>1$, the Sewing lemma for $\gamma\leq 1$ is an analytic tool which can nicely interact with
algebraic structures. The link between algebra, analysis and probability was already clear in Terry Lyons' seminal paper
\cite{MR1654527}, where geometric rough paths were defined in terms of tensor algebras, following the work of
Kuo-Tsai Chen on iterated integrals \cite{MR73174}. Later Massimiliano Gubinelli
pushed this beautiful interaction further by defining branched rough paths in terms of the Butcher-Connes-Kreimer
Hopf algebra \cite{MR2578445}. 

In this paper we show that the Sewing lemma, both for $\gamma>1$ and $\gamma\leq 1$,
allows to extend this framework to rough paths on a general (commutative graded connected locally finite) Hopf algebra. This includes
geometric and branched rough paths and also other notions introduced recently, like quasi-geometric rough paths \cite{Be20},
planarly branched rough paths \cite{MR4120383} and the (so far un-named) Hopf algebra of \cite[Section 6]{LOT21}.

In particular, the main application of the Sewing lemma for $\gamma\leq 1$ is the construction of rough paths over a 
H\"older path. The fact that one can always lift an $\alpha$-H\"older path for $\alpha\in(0,1)$ to a 
$\alpha$-geometric rough path has been long known, at least since the paper by Lyons-Victoir \cite{MR2348055}. However
their construction was based on a (beautiful) geometric and algebraic construction, which famously mentioned
the axiom of choice and was therefore considered as non-constructive. 

In \autoref{section:application_extension_theorem} we show that the Sewing lemma in the case $\gamma<1$ allows to
construct inductively in a simple way rough paths on a Hopf algebra over a $\alpha$-H\"older path for $\alpha\in(0,1)$,  up to
level $N:=\lfloor 1/\alpha\rfloor$ (higher levels are uniquely determined by the Sewing lemma for $\gamma>1$). Even
more, this construction is \emph{continuous} with respect to the relevant metrics. This was already known for
the second level of a geometric rough path, which can be reduced to an integral of the form $\int_{0}^t X_s \, {\rm d} Y_s$ in a case where the Young condition is not satisfied. Here we extend this to a general result.

Such a Sewing lemma in the regime $\gamma \in \left( 0, 1 \right]$ is proved in \autoref{thm:weak_Sewing} below.
Contrary to the case $\gamma > 1$, the constructed integral $I$ is not unique and is not defined by Riemann-type sums.
However, $I$ can be chosen to be linear in $A$ and continuous in an appropriate topology.
Note that in the context of regularity structures 
this is very close to the Reconstruction Theorem \cite{MR3274562, caravenna2020hairers} in the negative exponent case, where uniqueness is lost and different approximations are used, see Section \ref{sec:link} below for a discussion.

Another application of the Sewing lemma for $\gamma\in(0,1)$ is the bicontinuity of a natural bijection between
the set of rough paths on a Hopf algebra and a linear space of H\"older functions. This also extends to a general result
the content of \cite[Corollary~1.3]{MR4093955}, where a natural bijection between these
spaces, in the context of branched rough paths, had been constructed using a constructive Lyons-Victoir extension technique, 
but no proof of the continuity of this map was available. We mention that the continuity of this action and
of a Lyons-Victoir extension in the context of branched rough paths has been obtained using paraproducts in
\cite[Theorem 21]{BH1} and \cite[Corollary 3]{BH2}.

We note that other extensions results have been obtained since \cite{MR2348055}, see e.g.\ the renormalization method of 
\cite{MR2657813,MR3057184}, and \cite{MR2789583}, which uses probabilistic techniques in the case of the fractional Brownian motion.
Finally, we mention that it should be possible to extend our method to the Besov setting \cite{FS21}. Future extensions might
involve the framework of the Stochastic Sewing Lemma \cite{MR4089788}, the rough paths approach to non-commutative 
stochastic calculus \cite{MR3062538,BN21} and the rough paths approach to McKean-Vlasov equations \cite{DS21}.

\medskip
\noindent \textbf{Acknowledgements.}
We are very grateful to Cyril Labb\'e and Fran\-cesco Caravenna for their valuable insights and advice.

\section{Sewing lemmas}\label{section:weak_Sewing_lemma}

In the following we consider a time horizon $T > 0$.
We note $$\Delta_T^n:=\{(x_1,\ldots,x_n)\in\R^n: 0\leq x_1\leq \cdots \leq x_n\leq T\}.$$
We denote $\mathcal{C} \left( \Delta_T^n \right)$ the space of $\mathbb{R}$-valued continuous functions on $\Delta_T^n$; and $\mathcal{C}_n$ the space of $\mathbb{R}$-valued continuous functions on $[ 0, T ]^n$.

Note well that here the subscript $n$ corresponds to the dimension of the domain (and not the regularity of the function).

	\subsection{Presentation of the result}
Recall the usual Sewing lemma:
\begin{theorem}[{Sewing lemma for $\gamma>1$ \cite[Proposition~1]{MR2091358}, \cite[Lemma~2.1]{MR2261056}}]
\label{thm:usual_Sewing_lemma}
Let $\gamma > 1$ and $A \colon \Delta_T^2 \to \mathbb{R}$ be a continuous function such that
	\begin{equation}
		\left| A_{s, t} - A_{s, u} - A_{u, t} \right| \lesssim | t -  s |^{\gamma} ,
	\end{equation}
uniformly over $0\leq s\le u\le t\le T$.
Then there exists a unique function $I \colon [0,T] \to \mathbb{R}$ such that $I_{0} = 0$ and:
	\begin{equation}
		\left| I_t - I_s - A_{s, t} \right| \lesssim \left| t - s \right|^{\gamma} ,
	\end{equation}
uniformly over $0\le s\le t\le T$. Moreover $I$ is the limit of Riemann-type sums
\begin{equation}\label{eq:IPA}
	I_t = \lim_{|\mathcal{P}|\to 0} \ \sum_{i=0}^{\#\mathcal{P}-1}
	A_{t_{i} t_{i+1}}
\end{equation}
along \emph{arbitrary partitions} $\mathcal{P}$ of
$[0,T]$ with vanishing mesh
$|\mathcal{P}|\to 0$.
\end{theorem}
In \eqref{eq:IPA}, a \emph{partition} of the interval $[a,b]$ is a finite sequence of ordered points $\mathcal{P} = \{a = t_0 < t_1 < \ldots < t_k = b\}$; moreover we denote
$\#\mathcal{P} = k$ and $|\mathcal{P}| := \max_{i=1,\ldots, \#\mathcal{P}} |t_i-t_{i-1}|$.

\medskip
In this section, we establish the following theorem, extending the scope of the Sewing lemma to the regime $0 < \gamma \leq 1$.
\begin{theorem}[Sewing lemma for $0 < \gamma \leq 1$]\label{thm:simple_formulation_of_weak_Sewing}
Let $0 < \gamma \leq 1$ and $A \colon \Delta_T^2 \to \mathbb{R}$ be a continuous function such that
	\begin{equation}\label{eq:Sewing_assumption}
		\left| A_{s, t} - A_{s, u} - A_{u, t} \right| \lesssim | t -  s |^{\gamma} ,
	\end{equation}
uniformly over $0\leq s\le u\le t\le T$.
Then there exists a (non-unique) function $I \colon [0,T] \to \mathbb{R}$ such that $I_{0} = 0$ and:
	\begin{equation}\label{eq:Sewing_bound_in_simple_theorem}
		\left| I_t - I_s - A_{s, t} \right| \lesssim 			\begin{dcases}
				|t-s|^{\gamma} & \text{ if } 0 < \gamma < 1 , \\
				|t-s|\left( 1+\left| \log |t-s| \right| \right)  & \text{ if } \gamma = 1 , 
			\end{dcases}
	\end{equation}
uniformly over $0\le s\le t\le T$, and the map $A\mapsto I$ is linear.
\end{theorem}
In fact, we shall show a more general result in \autoref{thm:weak_Sewing} below, where we replace the right-hand side of \eqref{eq:Sewing_assumption} with $V (t-s)$ for general control functions $V$, see \autoref{def:weak_control_function}.
Note that then, \autoref{thm:simple_formulation_of_weak_Sewing} is an immediate consequence of \autoref{thm:weak_Sewing} and \autoref{ex:usual_control_function}.
Before turning to the statement and proof of \autoref{thm:weak_Sewing}, let us propose a few remarks.

\begin{remark}[Non-uniqueness]\label{remark:non_uniqueness}
It is straightforward to observe by the triangle inequality that any $I$ satisfying \eqref{eq:Sewing_bound_in_simple_theorem} for some $A \colon \Delta_T^2 \to \mathbb{R}$ and $\gamma \in (0,1)$ is determined up to a $\gamma$-H\"older function, i.e.\ if $I$ satisfies \eqref{eq:Sewing_bound_in_simple_theorem}, then $\tilde{I}$ also satisfies \eqref{eq:Sewing_bound_in_simple_theorem} if and only if $\tilde{I} - I$ is $\gamma$-H\"older. 
\end{remark}

\begin{remark}
In the usual Sewing lemma \autoref{thm:usual_Sewing_lemma}, $I$ is constructed as limit of Riemann-type sums, see \eqref{eq:IPA}.
In the case of \autoref{thm:simple_formulation_of_weak_Sewing}, we define $I$ with a different construction, namely
$I$ is defined explicitly (via a recursive formula) on the set of dyadic numbers and then
extended to $[0,T]$ by density.
Note that this approach is reminiscent of some constructive results of \cite{MR2348055, MR4093955}.
\end{remark}

\begin{remark}
The situation is similar to the setting of the Reconstruction Theorem \cite{MR3274562, caravenna2020hairers}, where different approximations are used depending on the sign of the exponent (named also $\gamma$), and where uniqueness is lost in the case of non-positive exponents. See \autoref{sec:link} below
for a discussion.
\end{remark}

	\subsection{The main technical result}

Now we turn to the statement and proof of our main technical result, \autoref{thm:weak_Sewing} below.
As in \cite{MR2372834}, we introduce a (new) notion of control function for our purposes.
We do not have a natural interpretation for this definition.
Rather, it corresponds to the quantity that appears in our proof below.

\begin{definition}[Control function]\label{def:weak_control_function}
If $V \colon [0,T] \to \mathbb{R}_+$, we set for $r, k_0 \in \mathbb{N}$:	
	\begin{align}
		\bar{V}_{(k_0)} \left(T 2^{-r} \right) & \coloneqq \left( k_0 + 1 \right) \sum\limits_{m = 0}^{+ \infty} 2^{m + 1} \, V \left( T 2^{- \left( r + m k_0 \right)} \right) \\
		& \quad + \sum\limits_{m = 0}^{+ \infty} \sum\limits_{k = 0}^{k_0} \sum\limits_{l = 1}^{r + m k_0 + k} 2^{ m + 1 - r - m k_0 - k + l} \, V \left( T 2^{- \left( l - 1 \right)} \right) .
	\end{align}

We say that $V$ is a $k_0$-control function if $V (0) = 0$, $V$ is increasing and for each $r \in \mathbb{N}$, $\bar{V}_{(k_0)} \left( T 2^{-r} \right) < + \infty$.
We extend $\bar{V}_{(k_0)} \colon [0,T] \to \mathbb{R}_+$ as follows: set $\bar{V}_{(k_0)} (0) \coloneqq 0$, and for $u \in \left( 0, 1 \right]$, set $\bar{V}_{(k_0)} \left( u \right) \coloneqq \bar{V}_{(k_0)} \left( T 2^{-r} \right)$ where $r \in \mathbb{N}$ is uniquely defined by $T 2^{- \left( r + 1 \right)} < u \leq T 2^{- r}$.
\end{definition}

\begin{example} \label{ex:usual_control_function}
Let $\gamma > 0$, $V \left( u \right) \coloneqq u^{\gamma}$.
Then for any integer $k_0 > \frac{1}{\gamma}$, $V$ is a $k_0$-control function and there exists a constant $C = C_{k_0, \gamma}$ such that:
	\begin{equation}
		\bar{V}_{(k_0)} \left( u \right) \leq 
			\begin{dcases}
				C u^{\gamma} & \text{ if } 0 < \gamma < 1 , \\
				C \left( 1 + | \log ( T ) | \right) \left( 1+ \left| \log \left( u \right) \right| \right) u & \text{ if } \gamma = 1 , \\
				C T^{\gamma - 1} u & \text{ if } \gamma > 1 .
			\end{dcases}
	\end{equation}
\end{example}

We need the operators $\delta$ defined as follows:
	\begin{enumerate}
		\item $\delta: \mathcal{C}_1 \to \mathcal{C}_2$: for $I \colon [0,T] \to \mathbb{R}$, we define 
		$$\delta I \colon [0,T]^{2} \to \mathbb{R}, \qquad \delta I_{s, t} \coloneqq I_t - I_s.$$
		\item $\delta: \mathcal{C}_2 \to \mathcal{C}_3$: for $A \colon [0,T]^2 \to \mathbb{R}$, we define 
		$$\delta A \colon [0,T]^{3} \to \mathbb{R}, \qquad \delta A_{s, u, t} \coloneqq A_{s, t} - A_{s, u} - A_{u, t}.$$
	\end{enumerate}
It is easy to see that $\delta\circ\delta=0$, so that these operators form a cochain complex, which is
moreover \emph{exact}: if $\delta Z=0$ for $Z\in \mathcal{C}_2$, then $Z=\delta z$
for some $z\in \mathcal{C}_1$ (see \cite{MR2091358}).

We will work with dyadic numbers: for $m \in \mathbb{N}$, denote $D_m \coloneqq \left\lbrace k 2^{- m}, k \in \left\llbracket 0, 2^m \right\rrbracket \right\rbrace$, and $D \coloneqq \bigcup_{m \in \mathbb{N}} D_m$.
Our main technical result is the following.

\begin{theorem} \label{thm:weak_Sewing}
Let $A \colon \Delta_T^2 \to \mathbb{R}$ be a function and $V$ be a $k_0$-control function for some $k_0 \in \mathbb{N}$. 
Assume that for all $0 \leq s \leq u \leq t \leq T$,
		\begin{equation}
			\left| 
			\left(\delta A \right)_{s, u, t} 
			\right| \leq V (t-s) .
		\end{equation}
Then:
	\begin{enumerate}[label = \textit{\arabic*.}, ref = \textit{\arabic*.}]
		\item\label{item:Sewing_on_D} (Sewing on dyadics) 			There exists $I \colon T D \to \mathbb{R}$ such that $I_0 = 0$ and for all $0 \leq s \leq t \leq T$ with $s/T, t/T \in D$, 	
		\begin{equation}\label{eq:Sewing_bound}
			\left| I_t - I_s - A_{s, t} \right| \leq \bar{V}_{(k_0)} (t-s) .
		\end{equation}
		
		\item\label{item:Sewing_on_0_1} (Sewing on $[0,T]$) Assume furthermore that $A$ is continuous and that $\bar{V}_{(k_0)} \left( u \right) \to_{u \to 0} 0$. 
		Let $W$ be a continous function such that $\bar{V}_{(k_0)} \leq W$.
		Then there exists $I \colon [0,T] \to \mathbb{R}$ such that $I_0 = 0$ and for all $0 \leq s \leq t \leq T$, 	
		\begin{equation}
			\left| I_t - I_s - A_{s, t} \right| \leq W (t-s) .
		\end{equation}
	\end{enumerate}
\end{theorem}

\begin{proof}
We prove the items in the announced order. 
Without loss of generality we can suppose that $T=1$: if $T \neq 1$, consider $\tilde{A} \colon \Delta_1^2 \to \mathbb{R}$ defined by $\tilde{A}_{s, t} \coloneqq A_{sT, tT}$.

\smallskip
{\it Proof of \ref{item:Sewing_on_D}}
First we introduce a sequence $u_{k, n}$ defined for $n \in \mathbb{N}$ and $k \in \left\llbracket 0, 2^n - 1 \right\rrbracket$ by the recurrence relations:
			\begin{equation}
				\begin{dcases}
					u_{0, 0} & \coloneqq 0 , \\
					u_{2 k, n + 1} & \coloneqq \frac{1}{2} u_{k, n} , \\
					u_{2 k + 1, n + 1} & \coloneqq \frac{1}{2} u_{k, n} + ( \delta A )_{k 2^{- n}, \left( k + \frac{1}{2} \right) 2^{- n}, \left( k + 1 \right) 2^{- n}} .
				\end{dcases}
			\end{equation}
		Now note that in order to define recursively $I$ on $D$, it is enough for each $n \in \mathbb{N}$ to define $I$ on elements of $D_n$ of the form $k 2^{-n}$ where $k = 2 l + 1$ is odd.
		Indeed, when $k$ is even, $k 2^{-n}$ belongs to $D_{n - 1}$.
		Thus, we set $I_0 \coloneqq 0$ and then recursively:
			\begin{equation}
				I_{\left( 2 l + 1 \right) 2^{-n}} \coloneqq I_{2 l 2^{- n}} + A_{ 2l 2^{- n}, \left( 2 l + 1 \right) 2^{-n}} + u_{2l, n} .
				\label{eq:recursive_definition_of_I}
			\end{equation}		
		(Notice that for each dyadic number $s \in D$, $I_s$ is actually defined as a finite linear combination of the $A_{u, v}$ and there is no question of convergence here. Note also that this definition is linear in $A$.)
		We claim that $R_{s, t} \coloneqq I_t - I_s - A_{s, t}$ can be expressed on consecutive elements of $D_n$ by:
			\begin{equation}
				R_{k 2^{-n}, \left( k + 1 \right)2^{-n}} = u_{k, n} .
				\label{eq:R_is_u_on_consecutive_dyadics}
			\end{equation}
		This is not clear a priori as \eqref{eq:recursive_definition_of_I} only establishes this property when $k = 2 l$ is even. However, when $k = 2 l + 1$ is odd, we argue by recurrence on $n$, and this is where we exploit the definition of $u$. Indeed, by definition of $R$, it holds that $\delta R = - \delta A$ so that:
			\begin{align}
				R_{\left( 2 l + 1 \right) 2^{-n}, \left( 2 l + 2 \right) 2^{- n}} &= R_{\left( 2 l \right) 2^{-n}, \left( 2 l + 2 \right) 2^{- n}} - R_{\left( 2 l \right) 2^{-n}, \left( 2 l + 1 \right) 2^{- n}} \\
				& \quad + ( \delta A )_{2 l 2^{- n}, \left( 2 l + 1 \right) 2^{- n}, \left( 2 l + 2 \right) 2^{- n}} .
			\end{align}
		
		By recurrence on $n$, $R_{\left( 2 l \right) 2^{-n}, \left( 2 l + 2 \right) 2^{- n}} = u_{l, n - 1}$.
		By \eqref{eq:recursive_definition_of_I}, we have $R_{\left( 2 l \right) 2^{-n}, \left( 2 l + 1 \right) 2^{- n}} = u_{2 l, n}$.
		By definition of $u$, $u_{2 l, n} = \frac{1}{2} u_{l, n - 1}$, thus:
			\begin{equation}
				R_{\left( 2 l + 1 \right) 2^{-n}, \left( 2 l + 2 \right) 2^{- n}} = \frac{1}{2} u_{l, n - 1} + ( \delta A )_{2 l 2^{- n}, \left( 2 l + 1 \right) 2^{- n}, \left( 2 l + 2 \right) 2^{- n}} ,
			\end{equation}
and by definition of $u$ this in turn equals $u_{2 l + 1, n}$, which establishes \eqref{eq:R_is_u_on_consecutive_dyadics}.
		
		Now let us turn to the Sewing bound \eqref{eq:Sewing_bound}.
		For $r, M \in \mathbb{N}$, set:
			\begin{equation}
				v_{r, M} \coloneqq \max\limits_{t, s \in D_M , 0 \leq t - s \leq 2^{- r}} \left| R_{s, t} \right| .
			\end{equation}		
		Note that when $r > M$, no distinct values $t, s \in D_M$ satisfy $0 \leq t - s \leq 2^{- r}$, so that $v_{r, M} = 0$.
		When $r = M$, the only elements $t, s \in D_r$ satisfying $0 \leq t - s \leq 2^{- r}$ are consecutive, i.e.\ of the form $s = k 2^{-r}$, $t = \left( k + 1 \right) 2^{- r}$ so that from \eqref{eq:R_is_u_on_consecutive_dyadics}, 
			\begin{equation}
				v_{r, r} = \max\limits_{k \in \left\llbracket 0, 2^r - 1 \right\rrbracket} \left| u_{k, r} \right| .
			\end{equation}
		According to the recursive definition of $u$, and the hypothesis on $A$,
			\begin{equation}
				v_{r, r} \leq \frac{1}{2} v_{r - 1, r - 1} + V \left( 2^{- \left( r - 1 \right)} \right) .
			\end{equation}
		Iterating and since $v_{0, 0} = 0$, we get:
			\begin{equation}
				v_{r, r} \leq \sum\limits_{l = 1}^r 2^{- \left( r - l \right)} V \left( 2^{- \left( l - 1 \right)} \right) .
				\label{eq:bound_on_consecutive_dyadics}
			\end{equation}			
	Now we establish a first (recursive) estimate of $v_{r, r + k}$ for $r, k \in \mathbb{N}$.
	Let $t, s \in D_{r + k}$ be such that $0 \leq t - s \leq 2^{- r}$.
	Note that if $t = s$ then $R_{s, t} = 0$, so that now we assume $t - s > 0$.
	Remembering that $\delta R = - \delta A$, we decompose:
		\begin{equation}
			R_{s, t} = R_{s, s_1} + R_{s_1, t_1} + R_{t_1, t} - ( \delta A )_{s_1, t_1, t} - ( \delta A )_{s, s_1, t} .
			\label{eq:decomposition_of_R}
		\end{equation}
where $s_1 \coloneqq \min \left( u \in D_{r + k - 1}, u \geq s \right)$, $t_1 \coloneqq \max \left( u \in D_{r + k - 1}, u \leq t \right)$.
	Note that $s_1, t_1$ are correctly defined, that $s_1, t_1 \in D_{r + k - 1} \subset D_{r + k}$, $s \leq s_1$, $t_1 \leq t$, and that $s_1 - s \leq 2^{- \left( r + k \right)}$, $t - t_1 \leq 2^{- \left( r + k \right)}$.
	Furthermore, since $t - s > 0$ and $t, s \in D_{r + k}$, it holds that $t - s \geq 2^{- \left( r + k \right)}$ and thus $D_{r + k - 1} \cap \left[ s, t \right] \neq \emptyset$, whence $s_1 \leq t_1$.
	Thus from \eqref{eq:decomposition_of_R} and the definition of $v$:
		\begin{equation}
			v_{r, r + k} \leq 2 v_{r + k, r + k} + v_{r, r + k - 1} + 2 V \left( 2^{- r} \right).
		\end{equation}	
	Recalling \eqref{eq:bound_on_consecutive_dyadics}, this yields:
		\begin{equation}
			v_{r, r + k} - v_{r, r + k - 1} \leq 2 \left( V \left( 2^{- r} \right) + \sum\limits_{l = 1}^{r + k} 2^{- \left( r + k - l \right)} V \left( 2^{- \left( l - 1 \right)} \right) \right) .
		\end{equation}	
	Summing from $k = 1$ to $K$ and reusing \eqref{eq:bound_on_consecutive_dyadics}, we obtain for $r, k \in \mathbb{N}$:
		\begin{equation} \label{eq:first_estimate_on_v_r_r+K}
			v_{r, r + K} \leq 2 K V \left( 2^{- r} \right) + 2 \sum\limits_{k = 0}^{K} \sum\limits_{l = 1}^{r + k} 2^{- \left( r + k - l \right)} V \left( 2^{- \left( l - 1 \right)} \right) .
		\end{equation}	
	Now fix $k_0 \in \mathbb{N}^{*}$, and let $r, M \in \mathbb{N}$ with $r \leq M$.
	Let $0 \leq t - s \leq 2^{- r}$ with $s, t \in D_{M}$.
	We consider several cases.
	
	If $r + k_0 \geq M$, then using \eqref{eq:first_estimate_on_v_r_r+K}:
		\begin{equation} \label{eq:estimate_of_v_case_1}
			v_{r, M} \leq 2 k_0 V \left( 2^{- r} \right) + 2 \sum\limits_{k = 0}^{k_0} \sum\limits_{l = 1}^{r + k} 2^{- \left( r + k - l \right)} V \left( 2^{- \left( l - 1 \right)} \right) .
		\end{equation}
	
	If $r + k_0 < M$ and $0 \leq t - s \leq 2^{- \left( r + k_0 \right)}$ then by definition of $v$:
		\begin{equation} \label{eq:estimate_of_v_case_2}
			\left| R_{s, t} \right| \leq v_{r + k_0, M} .
		\end{equation}
			
	If $r + k_0 < M$ and $2^{- \left( r + k_0 \right)} < t - s \leq 2^{- r}$, then we consider 
	\[
	s_1 \coloneqq \min \left( u \in D_{r + k_0}, u \geq s \right), 
	\quad t_1 \coloneqq \max \left( u \in D_{r + k_0}, u \leq t \right).
	\]
	Observe that $s_1, t_1$ are correctly defined, $s \leq s_1$, $t_1 \leq t$, $s_1, t_1 \in D_{r + k_0} \subset D_{M}$.
	Also, since $t - s > 2^{- \left( r + k_0 \right)}$, it holds that $D_{r + k_0} \cap \left[ s, t \right] \neq \emptyset$, whence $s_1 \leq t_1$.
	Finally, the definition of $s_1, t_1$ implies $s_1 - s \leq 2^{- \left( r + k_0 \right)}$, $t - t_1 \leq 2^{- \left( r + k_0 \right)}$.
	Thus from \eqref{eq:decomposition_of_R}:
		\begin{equation} \label{eq:estimate_of_v_case_3}
			\left| R_{s, t} \right| \leq 2 v_{r + k_0, M} + v_{r, r + k_0} + 2 V \left( 2^{- r} \right) .
		\end{equation}
	We denote:
		\begin{equation}
			W_{(k_0)} \left( r \right) \coloneqq 2 \left( k_0 + 1 \right) V \left( 2^{- r} \right) + 2 \sum\limits_{k = 0}^{k_0} \sum\limits_{l = 1}^{r + k} 2^{- \left( r + k - l \right)} V \left( 2^{- \left( l - 1 \right)} \right) ,
		\end{equation}			
so that combining \eqref{eq:estimate_of_v_case_1}, \eqref{eq:estimate_of_v_case_2}, \eqref{eq:estimate_of_v_case_3}, and \eqref{eq:first_estimate_on_v_r_r+K} gives for $r \leq M \in \mathbb{N}$:
		\begin{equation}
			v_{r, M} \leq 
				\begin{dcases}
					2 v_{r + k_0, M} + W_{(k_0)} \left( r \right) & \text{ if } r + k_0 < M , \\
					 W_{(k_0)} \left( r \right) & \text{ if } r + k_0 \geq M .
				\end{dcases}
		\end{equation}		
		Iterating this recursive estimate, we obtain:
			\begin{equation}
				v_{r, M} \leq \sum\limits_{m = 0}^{+ \infty} 2^m W_{(k_0)} \left( r + m k_0 \right) .
			\end{equation}					
		Note that from \autoref{def:weak_control_function}, the right-hand term equals $\bar{V}_{(k_0)} \left( 2^{- r} \right)$, and thus this implies:
			\begin{equation}
				\sup\limits_{t, s \in D , 2^{- \left( r + 1 \right)} < t - s \leq 2^{- r}} \left| R_{s, t} \right| \leq \bar{V}_{(k_0)} \left( 2^{- r} \right) .
			\end{equation}
		Recall that by definition, $\bar{V}_{(k_0)} \left( 2^{- r} \right) = \bar{V}_{(k_0)} (t-s)$ when $2^{- \left( r + 1 \right)} < t - s \leq 2^{- r}$.
		This is enough to conclude that for all $s, t \in D$ with $t \neq s$, $\left| R_{s, t} \right| \leq \bar{V}_{(k_0)} (t-s)$.
		When $t = s$, $| R_{t, t} | = | A_{t, t} | = | - \left(\delta A \right)_{t, t, t} | \leq V (0) = 0 \leq \bar{V}_{(k_0)} (0)$, whence the announced result.
		(Note that this is the only time in the proof where we use the assumption that $V (0) = 0$.)

\smallskip
{\it Proof of \ref{item:Sewing_on_0_1}} We extend $I$ on $[0,T]$ by density, setting for $t \in [0,T]$:
			\begin{equation}
				I_t \coloneqq \lim\limits_{\substack{s \to t \\ s \in D}} I_{s} .
			\end{equation}
		This is correctly defined, because for any choice of $\left( s_n \right)_{n \in \mathbb{N}} \in D^{\mathbb{N}}$ with $s_n \to t$, the Sewing estimate:
			\begin{equation}
				\left| I_{s_n} - I_{s_m} \right| \leq \left| A_{s_n\wedge s_m , s_n\vee s_m } \right| + \bar{V}_{(k_0)} \left( \left| s_n - s_m \right| \right) ,
			\end{equation}
		 implies that the sequence $\left( I_{s_n} \right)_{n \in \mathbb{N}}$ is Cauchy.
		If $W$ is a continuous function such that $\bar{V}_{(k_0)} \leq W$, then $\left| I_t - I_s - A_{s, t} \right| \leq W (t-s)$ for all $s, t \in [0,T]$.
\end{proof}

\begin{remark}[Non-locality] 
Note that our construction of $I$ in the theorem above is non-local, in the sense that $I_t - I_s$ may depend on the value of $A$ outside of $[ s, t ]^2$.
As an example, observe from the definition \eqref{eq:recursive_definition_of_I} of $I$ that $I_{3/4} = \frac{1}{2} A_{0, 1/2} + \frac{1}{2} A_{0, 1} + A_{1/2, 3/4} - \frac{1}{2} A_{1/2, 1}$.
Since $I_0 = 0$, it follows that $I_{3/4} - I_{0}$ depends on the value of $A_{0, 1}$.
\end{remark}

	\section{Continuity of the Sewing map}
\label{subsection:gubinelli_formulation}

The Sewing Lemma for $\gamma>1$ given in \autoref{thm:usual_Sewing_lemma} finds its main applications in the
theory of rough integration and rough differential equations \cite{MR2091358}. In this setting it is useful to 
introduce some function spaces and interpret the Sewing Lemma in terms of operators having nice properties on these spaces.
Then we introduce for all $\gamma>0$, $A\in \mathcal{C}(\Delta_T^2)$ and $B\in \mathcal{C}(\Delta_T^3)$ 
\[
\|A\|_{\mathcal{C}_2^{\gamma} ( \Delta_T^2 )}\coloneqq \sup\limits_{0\leq s< t\leq T} \frac{\left| A_{s, t} \right|}{| t - s |^{\gamma}}, \qquad \|B\|_{\mathcal{C}_3^{\gamma} ( \Delta_T^3 )}:=\sup\limits_{0\leq s <u< t \le T} \frac{\left| B_{s, u, t} \right|}{| t - s |^{\gamma}},
\]
with the associated normed spaces 
\begin{align}
\mathcal{C}_2^{\gamma}( \Delta_T^2 )&\coloneqq\{A\in \mathcal{C}(\Delta_T^2): \|A\|_{\mathcal{C}_2^{\gamma}( \Delta_T^2 )}<+\infty\}, \\
\mathcal{C}_3^{\gamma}( \Delta_T^3 )&\coloneqq\{B\in \mathcal{C}(\Delta_T^3): \|B\|_{\mathcal{C}_3^{\gamma}( \Delta_T^3 )}<+\infty\}.
\end{align}
Then we can reformulate the Sewing Lemma for $\gamma>1$ with a quantitative estimate as follows
\begin{theorem}[{Sewing map for $\gamma>1$, see \cite{MR2091358,MR2261056}}]
\label{thm:usual_Sewing_lemma2}
Let $\gamma > 1$ and $A\in \mathcal{C}(\Delta_T^2)$ satisfy
$\delta A\in \mathcal{C}_3^{\gamma}( \Delta_T^3 )$. Then there exists a unique $R\in \mathcal{C}_2^{\gamma}( \Delta_T^2 )$ such that
$\delta R=\delta A$. Moreover we have the estimate
	\begin{equation}\label{estimate>1}
		\| R \|_{\mathcal{C}_2^{\gamma}( \Delta_T^2 )} \leq C_\gamma \| \delta A \|_{\mathcal{C}_3^{\gamma}( \Delta_T^3 )}
	\end{equation}
	with $C_\gamma=(2^\gamma-2)^{-1}$.
\end{theorem}
The two statements in \autoref{thm:usual_Sewing_lemma} and \autoref{thm:usual_Sewing_lemma2} are essentially equivalent 
if we establish the correspondence $(A,R)\leftrightarrow(I,R)$ in such a way that
\[
I:[0,T]\to\mathbb{R}, \qquad I_0=0, \qquad \delta I=A-R.
\]
The uniqueness statement in \autoref{thm:usual_Sewing_lemma2} implies that if $A\in \mathcal{C}(\Delta_T^2)$ satisfies $\delta A=0$, then the corresponding $R$ is also equal to zero. We obtain that in \autoref{thm:usual_Sewing_lemma2} $R$ is in fact a function of $\delta A$ only. In other words, \autoref{thm:usual_Sewing_lemma2} allows to define the 
\emph{Sewing map} 
\[
\Lambda:\delta(\mathcal{C}(\Delta_T^2))\cap\mathcal{C}_3^{\gamma}( \Delta_T^3 )\to \mathcal{C}_2^{\gamma}( \Delta_T^2 ),
\qquad B=\delta A \mapsto R=\Lambda B,
\]
and the bound \eqref{estimate>1} yields the continuity property in the case $\gamma>1$
	\begin{equation}\label{estimateSewing>1}
		\| \Lambda B \|_{\mathcal{C}_2^{\gamma}( \Delta_T^2 )} \leq (2^\gamma-2)^{-1} \| B \|_{\mathcal{C}_3^{\gamma}( \Delta_T^3 )},
		\qquad \forall B\in \delta(\mathcal{C}(\Delta_T^2))\cap\mathcal{C}_3^{\gamma}( \Delta_T^2 ).
	\end{equation}
Since $\delta R=\delta A$, or equivalently since $\delta\circ\delta I=0$, we note that $\delta\Lambda=\text{Id}_{\delta(\mathcal{C}(\Delta_T^2))\cap\mathcal{C}_3^{\gamma}( \Delta_T^2 )}$, namely we can interpret $\Lambda$ as a right inverse of $\delta$. 

We now state the Sewing Lemma for $0<\gamma\leq 1$ with a quantitative estimate.
\begin{theorem}[{Sewing map for $0 < \gamma\leq1$}]
\label{thm:weak_Sewing2}
Let $0 < \gamma \leq 1$ and $A\in C(\Delta_T^2)$ with
$\delta A\in \mathcal{C}_3^{\gamma}( \Delta_T^3 )$. Then there exists $R\in \mathcal{C}(\Delta_T^2)$ such that $\delta R=\delta A$, the map $\delta A\mapsto R=:\Lambda(\delta A)$ is linear and for some constant $C_\gamma\geq 0$ 
we have the estimate
	\begin{alignat}{2}
		\| R \|_{\mathcal{C}_2^{\gamma}( \Delta_T^2 )} &\leq C_{\gamma} \| \delta A \|_{\mathcal{C}_3^{\gamma}( \Delta_T^3)} \quad && \text{if} \ \gamma<1,  \label{estimate<1} \\
	\sup\limits_{0 \leq s < t \leq T} \frac{|R_{s,t}|}{\left( 1+ \left| \log |t-s| \right| \right) |t-s|} & \leq C_1 \| \delta A \|_{\mathcal{C}_3^{1}( \Delta_T^3 )} \ && \text{if} \ \gamma=1.  \label{estimate=1} 
	\end{alignat}
One can actually bound the constants in \autoref{ex:usual_control_function} and observe that one can take:
	\begin{align*}
		C_\gamma &= \frac{2^{\gamma + 1}}{1 - 2^{1 - \gamma \left( \left\lfloor \frac{1}{ \gamma} \right\rfloor + 1 \right)}} \left( 2 + \left\lfloor \frac{1}{\gamma} \right\rfloor + \frac{2}{\left( 2^{1 - \gamma} - 1 \right) \left( 1 - 2^{- \gamma} \right)} \right) && \!\!\text{if} \ 0 <\gamma<1, \\ 
		C_1 &= \frac{96}{\log 2} \left( 1 + | \log  T  | \right) && \text{if} \ \gamma=1.
	\end{align*}
\end{theorem}
We have therefore \emph{another} Sewing map for $\gamma< 1$
\[
\Lambda:\delta(\mathcal{C}(\Delta_T^2))\cap\mathcal{C}_3^{\gamma}( \Delta_T^3 )\to \mathcal{C}_2^{\gamma}( \Delta_T^2 ),
\qquad B=\delta A \mapsto R=\Lambda B.
\]
The fact that $R$ is in fact a function of $\delta A$ follows from the construction in \autoref{thm:weak_Sewing} of the sequence $u_{k,n}$ 
which depends only on $\delta A$, see \eqref{eq:R_is_u_on_consecutive_dyadics}, \eqref{eq:decomposition_of_R}.
Linearity of the map $\delta A\mapsto R$ follows from the construction in the proof of \autoref{thm:weak_Sewing}.

The estimate \eqref{estimate<1} shows that also for $\gamma<1$ the Sewing map can be constructed
as a linear continuous operator $\Lambda:\delta(\mathcal{C}(\Delta_T^2))\cap\mathcal{C}_3^{\gamma}( \Delta_T^3 )\to \mathcal{C}_2^{\gamma}( \Delta_T^2 )$. For $\gamma=1$ this is also true, but the norm $\mathcal{C}_2^{\gamma}$ has to be
slightly modified in accordance with \eqref{estimate=1}, but
 we refrain from introducing a notation for this.

We define the classical space of $\beta$-H\"older functions on $[0,T]$ for $\beta\in(0,1]$:
\[
\mathcal{C}_1^{\beta}:=\{f\in C([0,T]): \|\delta f\|_{\mathcal{C}_2^{\beta}}<+\infty\}.
\]
Then the choice of $R$ in \autoref{thm:weak_Sewing2} can not be unique for $\gamma\leq 1$, since $R+\delta f$ for any $f\in\mathcal{C}_1^{\gamma}$ also satisfies the desired properties.

\begin{remark}
Note that there are \emph{two different} Sewing maps: one for $\gamma>1$ and one for $\gamma\leq 1$. We denote both by $\Lambda$, since it is always clear from the context which of the two is used.
\end{remark}

\subsection{The integration map}\label{sec:integ}
We have seen above that the Sewing map allows to define an \emph{integration map} $A\mapsto I$, for all
$A\in \mathcal{C}(\Delta_T^2)$ such that $\delta A\in \mathcal{C}_3^{\gamma}( \Delta_T^3 )$, where $I\in \mathcal{C}_1$ is defined by
			\begin{equation}\label{I1}
					I_0 = 0, \qquad
					\delta I - A = -\Lambda(\delta A)\in \mathcal{C}_2^{\gamma}( \Delta_T^2 ).
			\end{equation}
We note this linear map by 
\begin{equation}\label{I2}
\cI:\{A\in \mathcal{C}(\Delta_T^2): \delta A\in\mathcal{C}_3^{\gamma}( \Delta_T^3 )\}\to \mathcal{C}_1, \qquad 
\cI(A):=I.
\end{equation}
A natural question in this context is under which conditions $I$ belongs to a space of $\beta$-H\"older functions.
\begin{prop}\label{prop:Ih}
Let $\beta\in(0,1)$
and $A\in \mathcal{C}(\Delta_T^2)$ such that $\delta A\in\mathcal{C}_3^{\beta\wedge\gamma}( \Delta_T^3 )$. Then $I:=\cI(A)\in\mathcal{C}_1^{\beta}$ if and only if $A\in\mathcal{C}_2^{\beta}( \Delta_T^2 )$, and in this case
\begin{equation}
\|\delta I\|_{\mathcal{C}_2^{\beta}( \Delta_T^2 )} \leq \|A\|_{\mathcal{C}_2^{\beta}( \Delta_T^2 )}+C_{\beta\wedge\gamma}\|\delta A\|_{\mathcal{C}_3^{\beta\wedge\gamma}( \Delta_T^3 )}.
\end{equation}
\end{prop}
\begin{proof}
Let $A\in \mathcal{C}(\Delta_T^2)$ with $\delta A\in \mathcal{C}_3^{\gamma}( \Delta_T^3 )$. Since $\delta I = A-\Lambda(\delta A)$ we obtain for $\beta\in(0,1)$:
\[
\begin{split}
\|\delta I\|_{\mathcal{C}_2^{\beta}( \Delta_T^2 )} & \leq \|A\|_{\mathcal{C}_2^{\beta}( \Delta_T^2 )}+\|\Lambda(\delta A)\|_{\mathcal{C}_2^{\beta}( \Delta_T^2 )}
\\ & \leq \|A\|_{\mathcal{C}_2^{\beta}( \Delta_T^2 )}+\|\Lambda(\delta A)\|_{\mathcal{C}_2^{\beta\wedge\gamma}( \Delta_T^2 )}
\\ & \leq \|A\|_{\mathcal{C}_2^{\beta}( \Delta_T^2 )}+C_{\beta\wedge\gamma}\|\delta A\|_{\mathcal{C}_3^{\beta\wedge\gamma}( \Delta_T^3 )}.
\end{split}
\]
On the other hand we have
\[
\begin{split}
\|A\|_{\mathcal{C}_2^{\beta}( \Delta_T^2 )} & \leq \|\delta I\|_{\mathcal{C}_2^{\beta}( \Delta_T^2 )}+\|\Lambda(\delta A)\|_{\mathcal{C}_2^{\beta}( \Delta_T^2 )}
\\ & \leq \|\delta I\|_{\mathcal{C}_2^{\beta}( \Delta_T^2 )}+\|\Lambda(\delta A)\|_{\mathcal{C}_2^{\beta\wedge\gamma}( \Delta_T^2 )}
\\ & \leq \|\delta I\|_{\mathcal{C}_2^{\beta}( \Delta_T^2 )}+C_{\beta\wedge\gamma}\|\delta A\|_{\mathcal{C}_3^{\beta\wedge\gamma}( \Delta_T^3 )}.
\end{split}
\]
The proof is complete.
\end{proof}
In the terminology of \cite{caravenna2020hairers}:
\begin{itemize}
\item $\delta A\in\mathcal{C}_3^{\beta\wedge\gamma}( \Delta_T^3 )$ is a \enquote{coherence} condition,
\item $A\in\mathcal{C}_2^{\beta}( \Delta_T^2 )$ is a \enquote{homogeneity} condition.
\end{itemize}
See \autoref{sec:link} below for further discussions.

\subsection{Optimality of the case \texorpdfstring{$\gamma=1$}{gamma=1}.}
The following example shows that the growth rate $| t - s | | \log | t - s | |$ from \autoref{thm:simple_formulation_of_weak_Sewing} in the case $\gamma = 1$ is optimal.
Let us set $A_{s,t}:=|t-s|\log|t-s|$. Then for $s\leq u\leq t$,
\[
\begin{split}
\delta A_{s,u,t} &=|t-s|\log|t-s|-|u-s|\log|u-s|-|t-u|\log|t-u| 
\\ & = |t-s|\left( p\log\frac1p+(1-p)\log\frac1{1-p}\right)\in[0,(\log 2)|t-s|] ,
\end{split}
\]
where $p=\frac{|t-u|}{|t-s|}\in[0,1]$. Therefore $\delta A\in\mathcal{C}^1_3( \Delta_T^3 )$. Let us suppose now that we can improve the bound \eqref{eq:Sewing_bound_in_simple_theorem} in the sense that there exists $I\in \mathcal{C}_1$ such that
$I_0=0$ and
\[
\lim_{\varepsilon\downarrow 0} \sup_{|t-s|\le\varepsilon} \frac
{|I_t-I_t-A_{s,t}|}{|A_{s,t}|}=0.
\]
Then there exists $\delta_n\to 0$ such that
\[
\sup_{|t-s|\le\frac1n} 
{|I_t-I_t-A_{s,t}|}\leq \delta_n \frac{\log n}{n}.
\]
Let us set $t=i/n$ and $s=(i-1)/n$. Then $A_{s,t}=-\log n/n$ and
\[
-(1+\delta_n)\frac{\log n}{n}\leq I_{\frac in}-I_{\frac{i-1}{n}}\leq -(1-\delta_n)\frac{\log n}{n},
\]
and summing over $i$
\[
-(1+\delta_n)(\log n)t\leq I_t\leq -(1-\delta_n)(\log n)t, \qquad t=\frac i{n}.
\]
Since $t=\frac{im}{{nm}}$ for all $m\geq 1$, we obtain
\[
-(1+\delta_{nm})(\log n+\log m)t\leq I_t\leq -(1-\delta_{nm})(\log n+\log m)t, \qquad t=\frac i{n},
\]
Letting $m\to+\infty$ we obtain $I_t=-\infty$, which is a contradiction.

\subsection{Unordered times}

For the applications to rough paths of \autoref{section:application_extension_theorem}, it is important to extend the Sewing Lemmas to
functions $A\colon[0,T]^2\to\R$ rather than $A \colon \Delta_T^2 \to \R$.

For this, we consider continuous functions $A:[0,T]^2\to\R$ such that for some $\gamma>0$
\begin{equation}
		\left| A_{s, t} - A_{s, u} - A_{u, t} \right| \lesssim \left( | t - u | \vee | u - s | \right)^{\gamma} , \qquad s, u, t \in [0,T].
	\end{equation}
This implies in particular 
\begin{equation}
		\left| A_{s, t} - A_{s, u} - A_{u, t} \right| \lesssim| t - s | ^{\gamma} , \qquad 0\leq s\leq  u\leq t\leq T.
	\end{equation}
Thus, $\delta A \in \mathcal{C}_3^{\gamma} (\Delta_T^{3})$, and in particular we can apply \autoref{thm:usual_Sewing_lemma2} and \autoref{thm:weak_Sewing2}, so that we can consider $\Lambda ( \delta A ) \in \mathcal{C}_2^{\gamma} \left( \Delta_{T}^{2} \right)$ and $\mathcal{I} \left( A \right) \in \mathcal{C}_1$.
However, this gives us information on $\Lambda ( \delta A )_{s, t}$ only when $s \leq t$.
Still, we can recover information on $\Lambda ( \delta A )_{s, t}$ when $s > t$ by writing from \eqref{I1}:
	\begin{equation}
		\Lambda( \delta A )_{s, t} + \Lambda( \delta A )_{t, s} = - \delta A_{s, t, s} ,
	\end{equation}
\noindent so that for $\gamma\ne 1$
	\begin{equation}
		\frac{\left| \Lambda ( \delta A )_{s, t} \right|}{| t - s |^{\gamma}} \leq \left\| \Lambda ( \delta A ) \right\|_{\mathcal{C}_2^{\gamma} \left( \Delta_T^{2} \right)}  + \left| \delta A_{s, t, s} \right| \le (C_\gamma+1)\|\delta A\|_{\mathcal{C}_3^{\gamma} \left( \Delta_T^{3} \right)} ,
	\end{equation}
	with $C_\gamma$ as in \eqref{estimate>1}-\eqref{estimate<1}, respectively.

Thus, if we introduce for all $\gamma>0$, $A\in \mathcal{C}_2$ and $B\in \mathcal{C}_3$ 
\begin{align}
	\|A\|_{\mathcal{C}_2^{\gamma}}&\coloneqq \sup\limits_{s,t\in[0,T], \ s\ne t} \frac{\left| A_{s, t} \right|}{| t - s |^{\gamma}}, \\
	\|B\|_{\mathcal{C}_3^{\gamma}}&\coloneqq\sup\limits_{s,u,t\in[0,T], \ s\ne t} \frac{\left| B_{s, u, t} \right|}{\left( | t - u | \vee | u - s | \right)^{\gamma}},
\end{align}
with the associated normed spaces 
\begin{align}
\mathcal{C}_2^{\gamma}&\coloneqq\{A\in \mathcal{C}_2 : \|A\|_{\mathcal{C}_2^{\gamma}}<+\infty\}, \\
\mathcal{C}_3^{\gamma}&\coloneqq\{B\in \mathcal{C}_3 : \|B\|_{\mathcal{C}_3^{\gamma}}<+\infty\},
\end{align}
then the above argument combined with \autoref{thm:usual_Sewing_lemma2}, \autoref{thm:weak_Sewing2}, and \autoref{prop:Ih}, give the following sewing lemma for unordered times:
\begin{theorem}[Sewing lemma for $\gamma > 0$ and unordered times] \label{thm:weak_sewing_lemma_ordered}
Let $\gamma > 0$.
There exist linear maps 
	\begin{align}
		\Lambda & \colon \delta(\mathcal{C}_2)\cap\mathcal{C}_3^{\gamma} \to \mathcal{C}_2 , \\
		\cI & \colon \{A\in \mathcal{C}_2 \colon \delta A\in\mathcal{C}_3^{\gamma}\}\to \mathcal{C}_1 ,
	\end{align}
\noindent such that for $A \in \mathcal{C}_2$ with $\delta A \in \mathcal{C}_3^{\gamma}$:
	\begin{equation}\label{eq:properties_of_lambda_I}
		\delta \Lambda ( \delta A ) = \delta A, \qquad \mathcal{I} \left( A \right)_0 = 0, \qquad \delta \mathcal{I} \left( A \right) - A = - \Lambda ( \delta A ) .
	\end{equation}
Such maps are unique when $\gamma > 1$ but not when $0 < \gamma \leq 1$.
Furthermore:
\begin{enumerate}
	\item (Regularity of $\Lambda$) For $A \in \mathcal{C}_2$ with $\delta A \in \mathcal{C}_3^{\gamma}$
	\begin{align} 
		\| \Lambda ( \delta A ) \|_{\mathcal{C}_2^{\gamma}} &\leq (C_\gamma+1)\, \| \delta A \|_{\mathcal{C}_3^{\gamma}} && \text{if} \ \gamma \neq 1, \\
	\sup\limits_{s,t\in[0,T], s \neq t} 
	\frac{|\Lambda ( \delta A )_{s,t}|}{\left( 1+ \left| \log |t-s| \right| \right) |t-s|} &\leq (C_1+1) \| \delta A \|_{\mathcal{C}_3^{1}} && \text{if} \ \gamma=1.
	\end{align}
	\item (Regularity of $\mathcal{I}$) Let $0 < \beta <1$ then:
		\begin{equation}
			\|\delta \mathcal{I} ( A )\|_{\mathcal{C}_2^{\beta}} \leq \|A\|_{\mathcal{C}_2^{\beta}}+(C_{\beta\wedge\gamma}+1)\|\delta A\|_{\mathcal{C}_3^{\beta\wedge\gamma}}.
		\end{equation}
\end{enumerate}

\end{theorem}

\section{A continuous Lyons-Victoir extension}
\label{section:application_extension_theorem}

One of the most important ideas of rough paths theory is that one can build strong integration theories 
involving non-smooth paths $X \colon \left[ 0, T \right] \mapsto \mathbb{R}^d$, under the condition that one \enquote{enriches} $X$ with a collection of \enquote{iterated integrals of $X$ against itself}.
When $X$ is not smooth, such iterated integrals can not be defined classically: therefore one enriches $X$
rather with a collection of functions which retain some of the algebraic and analytic properties valid in the smooth case. The algebraic properties of the collection of its iterated integrals (also called its signature)
were first discovered by Chen \cite{MR73174}, and the theory of rough paths builds on this algebraic formalism.

In this section we show that the Sewing Lemma (\autoref{thm:weak_sewing_lemma_ordered}) for $\gamma\in(0,1)$ allows to
construct rough paths over H\"older paths $X \colon \left[ 0, T \right] \mapsto \mathbb{R}^d$ in a continuous way, extending
Lyons-Victoir's result \cite{MR2348055} to the setting of a general commutative graded conneted locally finite Hopf algebra (see below).

Let us discuss for example the case of (weakly) geometric rough paths. In this context, we have 
\begin{enumerate}
\item a parameter $\alpha\in(0,1)$, and the associated integer $N:=\lfloor 1/\alpha\rfloor$,
\item the finite set $S:=\cup_{n=1}^N\{1,\ldots,d\}^n$ endowed with the \emph{degree} function 
$$S\ni(i_1,\ldots,i_n) \mapsto |(i_1,\ldots,i_n)|:=n\in\{1,\ldots,N\},$$
\item a family $(\langle X_{\cdot},\tau\rangle)_{\tau\in S}\subset \mathcal{C}_2$ such that
for all $(i_1,\ldots,i_n)\in S$ and $s,u,t\in[0,T]$, one has
\begin{align}
|\langle X_{s,t},(i_1,\ldots,i_n)\rangle|&\lesssim |t-s|^{\alpha n}, \label{X1} \\
\delta \langle X,(i_1,\ldots,i_n)\rangle_{s,u,t}&= \sum_{k=1}^{n-1} \langle X_{s,u},(i_1,\ldots,i_k)\rangle \langle X_{u,t},(i_{k+1},\ldots,i_n)\rangle. \quad \ \label{X2}
\end{align}
\end{enumerate}
Then by \eqref{X1}-\eqref{X2} we have the estimate for all $\tau\in S$ and $s,u,t\in[0,T]$
\begin{equation}\label{X3}
|(\delta \langle X,\tau\rangle)_{s,u,t}|\lesssim \sum_{k=1}^{|\tau|-1} |u-s|^{\alpha k}|t-u|^{\alpha(|\tau|-k)} \lesssim
(|s-u|\vee|t-u|)^{\alpha |\tau|}.
\end{equation}
By definition, $\alpha|\tau|\le\alpha N\le 1$. The \emph{extension problem} is the following: suppose that 
\begin{enumerate}
\item $k\in\{1,\ldots,N-1\}$,
\item we have $(\langle X,\tau\rangle)_{\tau\in S, |\tau|\le k}\subset \mathcal{C}_2$ satisfying \eqref{X1}-\eqref{X2},
\item we have $\tau_0\in S$ with $|\tau_0|=k+1$.
\end{enumerate}
Is it then possible to find $\langle X,\tau_0\rangle\in \mathcal{C}_2$ which satisfies \eqref{X1}-\eqref{X2} as well?

It is now clear that the Sewing Lemma (\autoref{thm:weak_sewing_lemma_ordered}) for $\gamma < 1$ yields a positive
answer to this problem for all $k\in\{1,\ldots,N-1\}$ as long as $\alpha N<1$, by applying the Sewing map $\Lambda$ to
the right-hand side of \eqref{X2} for $\tau=\tau_0$ and calling the result $\langle X,\tau_0\rangle$.
If $\alpha N=1$ then for $k+1=N$ we can construct $\langle X,\tau_0\rangle$ similarly but in the right-hand side of
\eqref{X1} we have $|t-s||\log|t-s||$ instead of $|t-s|$.

We note that for $|\tau|=1$ formula \eqref{X2} implies $\delta \langle X,\tau\rangle=0$, namely $\langle X,\tau\rangle_{st}=f^\tau_t-f^\tau_s$
for some $\alpha$-H\"older function $f$. 
Therefore one starts with a family of $\alpha$-H\"older functions $(f^\tau)_{\tau\in S, |\tau|=1}$ on $[0,T]$,
and applying recursively \autoref{thm:weak_sewing_lemma_ordered} one can construct a family $(\langle X,\tau\rangle)_{\tau\in S}$
satisfying \eqref{X1}-\eqref{X2}. Moreover $(\langle X,\tau\rangle)_{\tau\in B}$ can be seen to depend in a continuous way
on $(f^\tau)_{\tau\in B, |\tau|=1}$. This is the basis of the \emph{extension theorem} that we prove in
\autoref{corollary:Lyons_Victoir_extension} below.

\subsection{Rough paths and Hopf algebras}
In what follows, we consider $H = \left( \bigoplus_{n \in \mathbb{N}} H_n , m, \mathds{1}, \Delta, \epsilon, S \right)$, a graded, connected ($H_0 = \mathrm{span} \left( \mathds{1} \right)$), locally finite ($0 < \dim \left( H_n \right) < + \infty$), commutative Hopf algebra on $\mathbb{R}$ with antipode $S$. The degree $|\tau|$ of $\tau\in\cup_n H_n$ is defined by $|\tau|=n$ if $\tau\in H_n$.

The structure of Hopf algebra relies on many properties of compatibility between the operations (see e.g.\ \cite{MR2290769, cartier2021} for more details on Hopf algebras).
For instance, we will use the compatibility of the product $m$ and the coproduct $\Delta$, which reads (where we note $\tau_{2, 3}$ to be the operator defined by $\tau_{2,3} \left( a \otimes b \otimes c \otimes d \right) \coloneqq a \otimes c \otimes b \otimes d$):
	\begin{equation}\label{eq:compatibility_product_coproduct}
\Delta \circ m =\left( m \otimes m \right) \circ \tau_{2, 3} \left( \Delta \otimes \Delta \right)  .
	\end{equation}

We will also work with the reduced coproduct $\Delta^{\prime} \colon H \to H \otimes H$ defined as
\[
\Delta^\prime \tau:=\Delta\tau-\tau\otimes\mathds{1}-\mathds{1}\otimes\tau.
\]
Recall that for $n \geq 1$, $\Delta^{\prime} \colon H_n \to \bigoplus_{p, q \geq 1, p + q = n} H_p \otimes H_q$, and that $\Delta^{\prime}$ satisfies the coassociativity property:
	\begin{equation}\label{eq:coassociativity_of_Delta}
		\left( \mathrm{Id} \otimes \Delta^{\prime} \right) \Delta^{\prime} = \left( \Delta^{\prime} \otimes \mathrm{Id} \right) \Delta^{\prime} .
	\end{equation}

All those constraints on $H$ are somehow quite restrictive.
In fact, the following result, due to Milnor and Moore, asserts that such an $H$ necessarily has a polynomial structure.

\begin{theorem}[Milnor-Moore, see {\cite{MR174052}}, {\cite[Theorem~3.8.3]{MR2290769}}]\label{thm:milnor_moore}
Let $H$ be a graded connected commutative Hopf algebra.
Then $H$ is a free polynomial algebra, whose indeterminates can be chosen to be homogeneous elements of $H$.
\end{theorem}

This polynomial structure will be useful for the construction of linear maps on $H$ that are also multiplicative, which we will call characters.
More precisely:
\begin{definition}[Characters]
Let $H$ be a real algebra on $\mathbb{R}$.
	\begin{enumerate}[label = \textit{\arabic*.}, ref = \textit{\arabic*.}]			
		\item Let $n \geq 1$. We say that a nonzero linear map $X \in L \left( H_{\leq n}, \mathbb{R} \right)$ is a truncated character of order $n$ on $H$ if for all $\sigma, \tau \in H_{\leq n}$ with $\sigma \tau \in H_{\leq n}$, it holds that $\left\langle X, \sigma \tau \right\rangle = \left\langle X, \sigma \right\rangle \left\langle X, \tau \right\rangle$.
			We note $G_n$ the set of truncated characters on $H$ of order $n$.
			
		\item We say that a nonzero linear map $X \in L \left( H, \mathbb{R} \right)$ is a character on $H$ if for all $\sigma, \tau \in H$, $\left\langle X, \sigma \tau \right\rangle = \left\langle X, \sigma \right\rangle \left\langle X, \tau \right\rangle$.
			We note $G$ the set of characters on $H$.
	\end{enumerate}
\end{definition}

We will be interested in rough paths, which we define now.
Heuristically, a rough path is a collection of biprocesses having the nice analytical and algebraic properties of iterated integrals.
\begin{definition}[Rough paths]
Let $H$ be a graded connected locally finite Hopf algebra on $\mathbb{R}$.
Let $\alpha \in \left( 0, 1 \right)$ and $N:=\lfloor1/\alpha\rfloor$.
A $\left( H, \alpha \right)$-rough path is a map $X \colon [0,T]^2 \to G_N$ such that:
\begin{enumerate}[label = \textit{\arabic*.}, ref = \textit{\arabic*.}]
\item\label{item:rough_path_def_1} for all $\tau \in \cup_{n \leq N} H_n$ we have $\left\langle X, \tau \right\rangle \in \mathcal{C}_2^{\alpha|\tau|}$, i.e.\
\begin{equation}\label{X4}
\sup\limits_{\substack{s \neq t \in [0,T]}} \dfrac{\left| \left\langle X_{s, t}, \tau \right\rangle \right|}{\left| t - s \right|^{\alpha |\tau|}} < + \infty.
\end{equation}

\item\label{item:rough_path_def_2} (Chen's relation) for all $\tau \in H_{\leq N}$ and $s, u, t \in [0,T]$, one has 
\begin{equation}\label{X5'}
\left\langle X_{s,t}, \tau \right\rangle = \left\langle X_{s, u} \otimes X_{u, t}, \Delta \tau \right\rangle.
\end{equation}
\end{enumerate}
We note $\mathrm{RP}^{\alpha} \left( H \right)$ the set of all $\left( H, \alpha \right)$-rough paths.
\end{definition}

Then our main result will be the following generalisation of the Lyons-Victoir extension theorem \cite{MR2348055}:
\begin{theorem}[A continuous extension]\label{corollary:Lyons_Victoir_extension}
Let $\alpha \in (0, 1)$ with $\alpha^{-1} \notin \mathbb{N}$. Let us set
\[
\mathscr{C}^{\alpha}_{H_1^*}:=\left\lbrace f:[0,T]\to H_1^*: \ f_0=0, \ \|f_t-f_s\|\lesssim |t-s|^\alpha \right\}.
\] 
There exists a continuous \enquote{extension map} $\mathscr{E} \colon \mathscr{C}^{\alpha}_{H_1^*}\to \mathrm{RP}^{\alpha} \left( H \right)$
such that for all $h \in H_1$ and $f \in \mathscr{C}^{\alpha}_{H_1^*}$, $\left\langle \mathscr{E} \left( f \right) , h \right\rangle = \langle f,h\rangle$.
\end{theorem}
In other words, every $\alpha$-H\"older path with values in the finite dimensional space $H_1^*$ can be lifted to a $\alpha$-rough
path on $H$, and this extension can be made in a continuous way. \autoref{corollary:Lyons_Victoir_extension} will be proved as
a corollary of \autoref{thm:P_is_a_homeomorphism} below.

\subsection{Extension to levels higher than \texorpdfstring{$N$}{N}}
By the definition of $\Delta^\prime$, the Chen relation \eqref{X5'} is equivalent to
\begin{equation}\label{X5}
\left( \delta \left\langle X, \tau \right\rangle \right)_{s, u, t} = \left\langle X_{s, u} \otimes X_{u, t}, \Delta^{\prime} \tau \right\rangle.
\end{equation}
Note that \eqref{X1}-\eqref{X2} are particular cases of \eqref{X4}-\eqref{X5}, respectively.
Then, a simple application of the Sewing Lemma (\autoref{thm:weak_sewing_lemma_ordered}) for $\gamma > 1$ yields that
\begin{prop}\label{prop:char}
Every $\left( H, \alpha \right)$-rough path $X \colon [0,T]^2 \to G_N$ has a 
unique extension $X \colon [0,T]^2 \to G$ satisfying \eqref{X4}-\eqref{X5} for all $\tau\in H$.
\end{prop}
\begin{proof}
Fix $\tau \in H_{N + 1}$.
Arguing as in \eqref{X3}, by \eqref{X4}-\eqref{X5} we have
\begin{equation}
|\left\langle X_{s, u} \otimes X_{u, t}, \Delta^{\prime} \tau \right\rangle| \lesssim (|s-u|\vee|t-u|)^{\alpha (N+1)},
\end{equation}
\noindent where by definition $\alpha(N+1)>1$. In order to apply \autoref{thm:weak_sewing_lemma_ordered} to $(s, u, t) \mapsto \left\langle X_{s, u} \otimes X_{u, t}, \Delta^{\prime} \tau \right\rangle$, we still have to check
that this function belongs to $\delta (\mathcal{C}_2)$. This is true since,
if we consider the function $F\in \mathcal{C}_2$, defined by $F_{s,t}:=\left\langle X_{0, s} \otimes X_{s, t}, \Delta^{\prime} \tau  \right\rangle$, then
\[
\begin{split}
(\delta F)_{s,u,t}
& = \langle X_{0,s} \otimes (\delta X)_{s,u,t}-(\delta X)_{0,s,u}\otimes X_{u,t}-X_{s,u}\otimes X_{u,t}, \Delta^{\prime} \tau \rangle
\\ & = \langle X_{0,s}\otimes X_{s,u}\otimes X_{u,t}, \left( \mathrm{Id} \otimes \Delta^{\prime} \right) \Delta^{\prime} \tau - \left( \Delta^{\prime} \otimes \mathrm{Id} \right) \Delta^{\prime}\tau\rangle \\
& \quad - \langle X_{s,u}\otimes X_{u,t}, \Delta^{\prime} \tau \rangle \\
&= -\langle X_{s,u}\otimes X_{u,t}, \Delta^{\prime} \tau \rangle,
\end{split}
\]
by coassociativity of $\Delta^{\prime}$, recall \eqref{eq:coassociativity_of_Delta}. Therefore by \autoref{thm:weak_sewing_lemma_ordered} for $\gamma=\alpha(N+1)>1$ there is a unique $\langle X,\tau\rangle\in\mathcal{C}_2^{\gamma}$ with the desired
properties \eqref{X4}-\eqref{X5}, given by $\langle X,\tau\rangle = \Lambda \left((s, u, t) \mapsto \left\langle X_{s, u} \otimes X_{u, t}, \Delta^{\prime} \tau \right\rangle \right)$.
By recurrence, one constructs in the same way $\langle X,\tau\rangle$ for all $\tau \in \oplus_{n \geq N + 1} H_n$.

However we still have to prove that this construction gives an element
of $G$, namely that $\langle X,\sigma\tau\rangle=\langle X,\tau\rangle\langle X,\sigma\rangle$ for all $\sigma, \tau \in H$. It is enough to assume $\sigma \tau \in \oplus_{n \geq N + 1} H_n$. 
By \eqref{eq:compatibility_product_coproduct} we have
\[
\Delta^{\prime} \left( \sigma \tau \right)=\left( m \otimes m \right) \circ \tau_{2, 3} \left( \Delta \sigma \otimes \Delta \tau - 1 \otimes \sigma \otimes 1 \otimes \tau - \sigma \otimes 1 \otimes \tau \otimes 1 \right).
\]
By \eqref{X5}
\begin{equation}
\begin{split}
&\left( \delta \left\langle X, \sigma \tau \right\rangle \right)_{s, u, t}  = \left\langle X_{s, u} \otimes X_{u, t}, \Delta^{\prime} \left( \sigma \tau \right) \right\rangle .
\end{split}
\end{equation}
On the other hand 
\begin{equation}
\begin{split}
&\delta \left( \left\langle X, \sigma \right\rangle\left\langle X, \tau \right\rangle \right)_{s, u, t}  = 
\\ & =\left\langle X_{s, t}, \sigma \right\rangle \left\langle X_{s, t}, \tau \right\rangle - \left\langle X_{s, u}, \sigma \right\rangle \left\langle X_{s, u}, \tau \right\rangle - \left\langle X_{u, t}, \sigma \right\rangle \left\langle X_{u, t}, \tau \right\rangle 
\end{split}
\end{equation}
and by the Chen relation \eqref{X5'}
\begin{alignat}{2}
&\left\langle X_{s, t}, \sigma \right\rangle \left\langle X_{s, t}, \tau \right\rangle	 && = \left\langle X_{s, u} \otimes X_{u, t} , \left( m \otimes m \right) \circ \tau_{2, 3} \left( \Delta \sigma \otimes \Delta \tau\right)\right\rangle,
\\ &\left\langle X_{s, u}, \sigma \right\rangle \left\langle X_{s, u}, \tau \right\rangle && = \left\langle X_{s, u} \otimes X_{u, t} , 1 \otimes \sigma \otimes 1 \otimes \tau \right\rangle,
\\ &\left\langle X_{u, t}, \sigma \right\rangle \left\langle X_{u, t}, \tau \right\rangle && = \left\langle X_{s, u} \otimes X_{u, t} , \sigma \otimes 1 \otimes \tau \otimes 1 \right\rangle.
\end{alignat}
We obtain $\delta \left\langle X, \sigma \tau \right\rangle=\delta \left( \left\langle X, \sigma \right\rangle\left\langle X, \tau \right\rangle \right)\in\mathcal{C}_3^{\alpha(N+1)}$. Since 
$\left\langle X, \sigma \tau \right\rangle$ and $\left\langle X, \sigma \right\rangle\left\langle X, \tau \right\rangle$ are both in $\mathcal{C}_2^{\gamma}$, we conclude by \autoref{thm:weak_sewing_lemma_ordered}.
\end{proof}

Since $H$ is locally finite, the same can be said of the corresponding set of indeterminates.
Hence from \autoref{thm:milnor_moore}, for each $n \geq 1$, there exists a finite set $B_n \subset H_n$  such that:
	\begin{equation}
		H = \mathbb{R} \left[ \mathds{1} \cup \bigcup_{n \geq 1} B_n \right] .
	\end{equation}
We denote 
\[
H_{\leq n}:=\oplus_{k=0}^n H_k, \qquad B_{\le n}:=\cup_{k=1}^n B_k.
\]
We also note $B \coloneqq \bigcup_{n \geq 1} B_n$, to be the set of generating monomials, so that $H = \mathbb{R} \left[ \mathds{1} \cup B \right]$. 
The Milnor-Moore theorem asserts the existence of such a basis of generating monomials, but $B$ is neither unique nor canonical in general, see \autoref{ex:basis_of_some_hopf_algebras} below for some examples in usual cases.
We note that $B_1$ is always a linear basis for $H_1$.

\begin{example}[Examples of $(H,\alpha)$-rough paths] \label{ex:basis_of_some_hopf_algebras}
Here are some examples of Hopf algebras to which the above applies:
\begin{enumerate}
	\item The shuffle algebra on the alphabet $\{1,\ldots,d\}$, and the theory of (weakly) geometric rough paths on $\R^d$; in this case, one can take as a basis $B$ the set of Lyndon words \cite{MR2035110}.
	\item The Butcher-Connes-Kreimer algebra of rooted forests with nodes decorated by $\{1,\ldots,d\}$, and the theory of branched rough paths on $\R^d$; in this case, one can take as basis $B$ the set of trees \cite{MR2578445}.
\end{enumerate}
See \cite{MR3300969}, \cite[Section~4]{MR4093955} and \cite{MR3949967} for more details on geometric and branched rough paths.
\begin{enumerate}\setcounter{enumi}{2}
	\item Quasi-shuffle algebras and quasi-geometric rough paths \cite{Be20}.
	\item The Hopf algebra of Lie group integrators and planarly branched rough paths \cite{MR4120383}.
	\item The recently introduced (and so far un-named) Hopf algebra of \cite[Section 6]{LOT21}.
\end{enumerate}
\end{example}

\subsection{An isomorphism}
What characterizes a $(H,\alpha)$-rough path? Sin\-ce $B$ is a set of generators of $H$ as an algebra,
by linearity and multiplicativity, $X \in \mathrm{RP}^{\alpha} \left( H \right)$ is uniquely determined by the values of $\left\langle X, h \right\rangle$ for $h \in B$. In fact, by \autoref{prop:char}, the values of $\left\langle X, h \right\rangle$ for $h \in \cup_{n>N} H_n$ are uniquely determined by the values for $h \in \cup_{n\le N} B_n$. 

It remains to characterize the family $(\left\langle X, h \right\rangle)_{h \in B_{\le N}}$. For all $\tau \in B_{\le N}$, by the Chen relation \eqref{X5}, $\delta \langle X, \tau \rangle$ is characterized by 
$(\left\langle X, h \right\rangle)_{h \in B_{\le |\tau|-1}}$. However, since $\alpha |\tau|\leq 1$, the
function $\langle X, \tau \rangle$ is determined by $\delta \langle X, \tau \rangle\in\mathcal C^{\alpha|\tau|}_3$ only up to a
$(\alpha |\tau|)$-H\"older function $f^\tau:[0,T]\to\R$ such that $f^\tau_0=0$, see \autoref{remark:non_uniqueness}. 

It was indeed shown in \cite{MR4093955} that, in the case of branched rough paths, 
$\mathrm{RP}^{\alpha} \left( H \right)$ is in 
a bijective correspondence with
	\begin{equation}
		\mathscr{C}_{B}^{\alpha} \coloneqq \left\lbrace \left( f^{h} \right)_{h \in B_{\leq N}} : \text{for all } h \in B_{\leq N}, \ f^h \in \mathcal{C}^{\left| h \right| \alpha}_1 \text{ and } f^h_0 = 0 \right\rbrace .
	\end{equation}
While the construction of \cite{MR4093955}, restricted to the Butcher-Connes-Kreimer Hopf algebra, used the Hairer-Kelly map \cite{MR3300969} and a
Lyons-Victoir extension technique, the approach of this paper, based on the Sewing Lemma
(\autoref{thm:weak_sewing_lemma_ordered}) for $\gamma < 1$, is more elementary and more general at the same time. For a general
(graded, connected, locally finite, commutative) Hopf algebra $H$ we introduce the map
	\begin{equation} \label{eq:bijection_rough_paths_functions}
		\writefun{\mathscr{P}}{\mathrm{RP}^{\alpha} \left( H \right)}{\mathscr{C}_{B}^{\alpha}}{X}{\mathscr{P} \left( X \right) \coloneqq \left( \cI \left( \left\langle X, h \right\rangle \right) \right)_{h \in B_{\leq N}} ,}
	\end{equation}
	where $\cI$ is the integration map defined in \autoref{thm:weak_sewing_lemma_ordered} and $N=\lfloor1/\alpha\rfloor$.
We prove that $\mathscr{P}$ is bijective and furthermore we show that $\mathscr{P}$ is bicontinuous with respect to the distances defined for $X, Y \in \mathrm{RP}^{\alpha} \left( H \right)$ resp.\ $f, g \in \mathscr{C}_{B}^{\alpha}$ by:
	\begin{equation}
		\begin{dcases}
			d_{\mathrm{RP}^{\alpha} \left( H \right)} \left( X, Y \right) & \coloneqq \sum\limits_{h \in B_{\le N}} \left\| \left\langle X - Y, h \right\rangle \right\|_{\mathcal{C}_2^{| h | \alpha}} , \\
			d_{\mathscr{C}_{B}^{\alpha}} \left( f, g \right) & \coloneqq \sum\limits_{h \in B_{\leq N}} \left\| g^h - f^h \right\|_{\mathcal{C}_1^{| h | \alpha}} .
		\end{dcases}
	\end{equation}

The main result of this section is the following.

\begin{theorem} \label{thm:P_is_a_homeomorphism}
Let $\alpha \in \left( 0, 1 \right)$ with $\alpha^{-1} \notin \mathbb{N}$.
The map $\mathscr{P}$ in \eqref{eq:bijection_rough_paths_functions} 
is a locally bi-Lipschitz homeomorphism between $\left( \mathrm{RP}^{\alpha} \left( H \right), d_{\mathrm{RP}^{\alpha} \left( H \right)} \right)$ and $\left( \mathscr{C}_{B}^{\alpha}, d_{\mathscr{C}_{B}^{\alpha}} \right)$.
\end{theorem}

Before proving this result, let us discuss some corollaries. First we prove \autoref{corollary:Lyons_Victoir_extension}.

\begin{proof}[Proof of \autoref{corollary:Lyons_Victoir_extension}]
It suffices to identify $\mathscr{C}^{\alpha}_{H_1^*}$ with $\mathscr{C}^{\alpha}_{B_1}$, where
\[
 \mathscr{C}^{\alpha}_{B_1}:=\left\lbrace (f^h)_{h\in B_1}: \  \text{for all } h \in B_{1}, \ f^h \in \mathcal{C}^{\alpha}_1 \text{ and } f^h_0 = 0 \right\} 
\]
and then take $\mathscr{E} = \mathscr{P}^{-1}\circ \iota$, where $\iota:\mathscr{C}_{B_1}^{\alpha}\to\mathscr{C}_{B}^{\alpha}$ is the canonical injection obtained by defining $f^h:=0$ for all
$h\in B_{\le N}\setminus B_1$.
\end{proof}

\begin{corollary}[An action on Rough Paths, see \cite{MR4093955}]\label{corollary:action_on_rough_paths}
Let $\alpha \in \left( 0, 1 \right)$ with $\alpha^{-1} \notin \mathbb{N}$ and define:
	\begin{equation} \label{eq:definition_action}
		\writefun{T}{\mathscr{C}_B^{\alpha} \times \mathrm{RP}^{\alpha} \left( H \right)}{\mathrm{RP}^{\alpha} \left( H \right)}{\left( g, X \right)}{gX \coloneqq \mathscr{P}^{-1} \left( g + \mathscr{P} \left( X \right) \right).}
	\end{equation}
Then:
	\begin{enumerate}[label = \textit{\arabic*.}, ref = \textit{\arabic*.}]
		\item\label{item:action_1} $T$ is an action: for each $g, g^{\prime} \in \mathscr{C}^{\alpha}_B$ and $X \in \mathrm{RP}^{\alpha} \left( H \right)$, $g^{\prime} \left( g X \right) = \left( g + g^{\prime} \right) X$.
		\item\label{item:action_2} $T$ is free and transitive: for each $X, X^{\prime} \in \mathrm{RP}^{\alpha} \left( H \right)$, there exists a unique $g \in \mathscr{C}^{\alpha}_B$ such that $X^{\prime} = g X$.
		\item\label{item:action_3} $T$ is continuous.
		\item\label{item:action_4} Let $g \in \mathscr{C}^{\alpha}_B$ be such that there exists a unique $h \in B$ such that $g^h \neq 0$.
			Then $\left\langle gX, h \right\rangle = \left\langle X, h \right\rangle + \delta g^h$.			
			Furthermore, let $\tilde{B} \subset B \setminus \left\lbrace h \right\rbrace$ be a set of monomials such that $\Delta^{\prime} \left( \mathbb{R} \left[ \tilde{B} \right] \right) \subset \mathbb{R} \left[ \tilde{B} \right] \otimes \mathbb{R} \left[ \tilde{B} \right]$.
			Then for any $\tau \in \mathbb{R} \left[ \tilde{B} \right]$, $\left\langle gX, \tau \right\rangle = \left\langle X, \tau \right\rangle$.
	\end{enumerate}
\end{corollary}

\begin{proof}
Items \ref{item:action_1}, \ref{item:action_2}, \ref{item:action_3} are straightforward from \eqref{eq:definition_action} and \autoref{thm:P_is_a_homeomorphism}.
Also, one obtains \ref{item:action_4} recursively by using the explicit definition of $T$ in \eqref{eq:definition_action} and the recursive construction of $\mathscr{P}^{-1}$ from the proof of \autoref{thm:P_is_a_homeomorphism}.
\end{proof}

\begin{remark}[Link with \cite{MR4093955}]
When $H$ is the Butcher-Connes-Kreimer algebra, $\mathrm{RP}^{\alpha} \left( H \right)$ corresponds to the set of branched rough paths, also denoted $\mathrm{BRP}^{\alpha}$ in \cite{MR4093955}.
In that case, let $h \in H$ and $\tilde{B}$ be the set of trees not containing $h$.
By definition of the coproduct from admissible cuts, $\Delta^{\prime} \left( \mathbb{R} \left[ \tilde{B} \right] \right) \subset \mathbb{R} \left[ \tilde{B} \right] \otimes \mathbb{R} \left[ \tilde{B} \right]$ and \autoref{corollary:action_on_rough_paths} generalises \cite[Theorem~1.2]{MR4093955}.
\end{remark}

\subsection{Proof}
Now we turn to the proof of \autoref{thm:P_is_a_homeomorphism}.
\begin{proof}[Proof of \autoref{thm:P_is_a_homeomorphism}]
We organise this proof in five different steps.

\smallskip
 \textit{Step 1: $\mathscr{P}$ is well-defined.}
Indeed, if $X \in \mathrm{RP}^{\alpha} \left( H \right)$ and $h \in B_{\leq N}$, by \eqref{X4}-\eqref{X5} 
$\langle X , h \rangle\in\mathcal{C}^{|h|\alpha}_2$ and $\delta\langle X , h \rangle\in\mathcal{C}^{|h|\alpha}_3$, so that by \autoref{thm:weak_sewing_lemma_ordered} we obtain $\cI(\left\langle X , h \right\rangle)\in \mathcal{C}_1^{|h|\alpha}$.

\smallskip
 \textit{Step 2: continuity of $\mathscr{P}$.}
Let $X, Y \in \mathrm{RP}^{\alpha} \left( H \right)$, and $h \in B_{\leq N}$.
By linearity of $\cI$, $\left( \mathscr{P} \left( X \right) \right)^h - \left( \mathscr{P} \left( Y \right) \right)^h = \cI \left( \left\langle X - Y , h \right\rangle \right)$.
Now by continuity of $\cI$ (\autoref{thm:weak_sewing_lemma_ordered}),
	\begin{equation}
		\left\| \cI \left( \left\langle X - Y , h \right\rangle \right) \right\|_{\mathcal{C}_1^{| h | \alpha}} \lesssim \left\| \left\langle X - Y , h \right\rangle \right\|_{\mathcal{C}_2^{| h | \alpha}} + \left\| \delta \left\langle X - Y , h \right\rangle \right\|_{\mathcal{C}_3^{|h| \alpha}} .
	\end{equation}

Observe that $\left\| \delta \left\langle X - Y , h \right\rangle \right\|_{\mathcal{C}_3^{|h| \alpha}} \leq 3 \left\| \left\langle X - Y , h \right\rangle \right\|_{\mathcal{C}_2^{| h | \alpha}}$.
Also, by definition, one has $\left\| \left\langle X - Y , h \right\rangle \right\|_{\mathcal{C}_2^{| h | \alpha}} \leq d_{\mathrm{RP}^{\alpha} \left( H \right)} \left( X, Y \right)$, thus we obtain $d_{\mathscr{C}_B^{\alpha}} \left( \mathscr{P} \left( X \right), \mathscr{P} \left( Y \right) \right) \lesssim d_{\mathrm{RP}^{\alpha} \left( H \right)} \left( X, Y \right)$, which establishes the Lipschitz continuity of $\mathscr{P}$.

\smallskip
 \textit{Step 3: injectivity of $\mathscr{P}$.} In this step, we use the following general fact: if $A, \tilde{A} \in \mathcal{C}_2$ are such that $\delta A = \delta \tilde{A}$ and $\mathcal{I} ( A ) = \mathcal{I} ( \tilde{A} )$, then $A = \tilde{A}$.
Indeed, this follows immediately from the defining property \eqref{I1} of $\cI$: $$\delta\circ\cI={\rm id}-\Lambda\circ\delta.$$
Let $X, Y \in \mathrm{RP}^{\alpha}$ and assume $\mathscr{P} \left( X \right) = \mathscr{P} \left( Y \right)$.
That is, for all $h \in B$, $\mathcal{I} ( \left\langle X, h \right\rangle ) = \mathcal{I} ( \left\langle Y, h \right\rangle )$.
We prove by recurrence that for all $h \in B$, $\left\langle X, h \right\rangle = \left\langle Y , h \right\rangle$.
If $h \in B_1$, then by definition of the reduced coproduct $\Delta^{\prime}$ and Chen's relation, it holds that $\delta \left\langle X , h \right\rangle = \delta \left\langle Y , h \right\rangle = 0$ and thus from the remark just above, $\left\langle X, h \right\rangle = \left\langle Y, h \right\rangle$, whence the initialisation of the recurrence.
Now let $n \in \mathbb{N}$ and assume that $X$ and $Y$ coincide on $B_{\leq n}$.
Let $h \in B_{n + 1}$.
Once again using Chen's relation, the definition of the reduced coproduct, and the fact that $X$ and $Y$ are characters, one gets by the recurrence hypothesis that $\delta \left\langle X , h \right\rangle = \delta \left\langle Y , h \right\rangle$.
Applying again the remark just above, it follows that $\left\langle X, h \right\rangle = \left\langle Y, h \right\rangle$.
By recurrence, this yields the injectivity of $\mathscr{P}$.

\smallskip
 \textit{Step 4: surjectivity of $\mathscr{P}$.}
In this step, we construct a right-inverse $\mathscr{P}^{-1}$ of $\mathscr{P}$ (which is also an actual inverse thanks to the previous step).
For this purpose, we proceed recursively.
Fix $f \in \mathscr{C}_B^{\alpha}$.
We first define $X \coloneqq \mathscr{P}^{-1} \left( f \right)$ on $H_1$ by setting for $h \in B_1$: $\left\langle X, h \right\rangle \coloneqq f^h$, then extending $X$ to $H_1$ by linearity.
Now let $n \geq 1$ and assume that $X$ is defined on $H_{\leq n}$.
We start by defining $X$ on the monomials i.e.\ on $B_{n + 1}$.
Fix $h \in B_{n + 1}$ and let us show that there exists $\left\langle X, h \right\rangle \colon [0,T]^2 \to \mathbb{R}$ such that:
	\begin{enumerate}
		\item $\left| \left\langle X_{s, t}, h \right\rangle\right| \lesssim \left| t - s \right|^{\alpha \left( n + 1 \right)}$,
		\item $\left( \delta \left\langle X, h \right\rangle \right)_{s, u, t} = \left\langle X_{s, u} \otimes X_{u, t}, \Delta^{\prime} h \right\rangle$.
	\end{enumerate}
For this purpose, arguing as in the proof of \autoref{prop:char} we consider the function $F\in \mathcal{C}_2$ defined for $s, t \in [0,T]$ by:
	\begin{equation}
		F_{s, t} \coloneqq \left\langle X_{0, s} \otimes X_{s, t}, \Delta^{\prime} h \right\rangle.
	\end{equation}
	Then 
by the coassociativity of $\Delta^{\prime}$, the Chen relation and the recurrence hypothesis, 
for $s, u, t \in [0,T]$ we have $\left( \delta F \right)_{s, u, t} = -\left\langle X_{s, u} \otimes X_{u, t}, \Delta^{\prime} h \right\rangle$.
Also, from the definition of $\Delta^{\prime}$ and the analytic constraint \eqref{X4} on $(\langle X,\tau\rangle)_{\tau\in H_n}$, one observes that $\delta F \in \mathcal{C}_3^{\left( n + 1 \right) \alpha}$.
Then, from the properties of the Sewing map, it suffices to set:
	\begin{equation}\label{eq:recursive_definition_of_X_against_monomials}
		\left\langle X, h \right\rangle \coloneqq -\Lambda \left( \delta F \right) + f^h .
	\end{equation}

Note that \eqref{eq:recursive_definition_of_X_against_monomials} can be rewritten as $\cI \left( \left\langle X, h \right\rangle \right) = f^h$, i.e.\ $\mathscr{P} \left( X \right)^h = f^h$, as wanted for the construction of the inverse.
Now it remains to suitably extend $X$ to the whole of $H_{\leq n + 1}$, which we do \enquote{polynomially}: if $P$ is a polynomial, we set 
	\begin{equation}
		\left\langle X, P \left( \left( h \right)_{h \in B_{\leq n + 1}} \right) \right\rangle \coloneqq P \left( \left( \left\langle X, h \right\rangle \right)_{h \in B_{\leq n + 1}} \right) .
	\end{equation}
It is straightforward to observe that this correctly defines an element of $G_{n + 1}$ and enforces the estimate $\left| \left\langle X_{s, t}, \tau \right\rangle \right| \lesssim \left| t - s \right|^{\alpha \left( n + 1 \right)}$ for all $\tau \in H_{n + 1}$.
To conclude the recursive step, it suffices to establish Chen's relation on $H_{n + 1}$, knowing that it is satisfied on $H_n$ and on $B_{n + 1}$ (by the construction just above).
Observe that if Chen's relation is satisfied on $\tau$ and $\sigma$, then it is also satisfied on $\tau + \sigma$.
Now because we define $X$ polynomially, it suffices to prove that if Chen's relation is satisfied for $\sigma \in H_l$ and $\tau \in H_k$ where $k + l \leq n + 1$, $1 \leq k, l \leq n$, then it is also satisfied for $\sigma \tau$. For this, it is enough to prove that $\delta\left\langle X, \sigma \tau \right\rangle=
\delta \left( \left\langle X, \sigma \right\rangle\left\langle X, \tau \right\rangle \right)$. This can be done arguing as in \autoref{prop:char}. (While in \autoref{prop:char} we had to verify that the extension was multiplicative, here this property is enforced by the \enquote{polynomial} definition).
This concludes the recursive step, so that we have constructed $X \eqqcolon \mathscr{P}^{-1} \left( f \right) \in \mathrm{RP}^{\alpha} \left( H \right)$.

\smallskip
 \textit{Step 5: continuity of $\mathscr{P}^{-1}$.}
We first show that $\mathscr{P}^{-1}$ maps bounded sets to bounded sets.
For this purpose, let $C > 0$, we shall show that there exists $C^{\prime} > 0$ such that for all $f \in \mathscr{C}_B^{\alpha}$ with $\sum_{h \in B} \| f^h \|_{\mathcal{C}_1^{|h| \alpha}} \leq C$, one has $\sum_{h \in B} \left\| \left\langle \mathscr{P}^{-1} \left( f \right), h \right\rangle \right\|_{\mathcal{C}_2^{|h| \alpha}} \leq C^{\prime}$.
We proceed recursively: when $h \in B_1$, by construction of $\mathscr{P}^{-1}$, it holds that $\left\langle \mathscr{P}^{-1} \left( f \right), h \right\rangle = f^h$, and thus $\left\| \left\langle \mathscr{P}^{-1} \left( f \right), h \right\rangle \right\|_{\mathcal{C}_2^{|h| \alpha}} = \left\| f \right\|_{\mathcal{C}_1^{|h| \alpha}} \leq C$.
Now fix $n \geq 1$ and assume that for all $b \in B_{\leq n}$,  $\left\| \left\langle \mathscr{P}^{-1} \left( f \right) , b \right\rangle \right\|_{\mathcal{C}_2^{| b | \alpha}} \leq C^{\prime}$ for some constant $C^{\prime}$.
Let $h \in B_{n + 1}$, so that by the construction above:
	\begin{equation} \label{eq:recursive_construction_for_the_inverse_of_P}
		\left\langle \mathscr{P}^{-1} \left( f \right) , h \right\rangle = \Lambda \left( \delta \left\langle \mathscr{P}^{-1} \left( f \right), h \right\rangle \right) + f^h .
	\end{equation}
	
Using Chen's relation, we obtain a decomposition:
	\begin{equation}
		\delta \left\langle \mathscr{P}^{-1} \left( f \right), h \right\rangle_{s, u, t} = \sum\limits_{\substack{\sigma \in H_k, \tau \in H_l \\ k, l \geq 1, k + l = |h|}} \left\langle \mathscr{P}^{-1} \left( f \right)_{s, u} , \sigma \right\rangle \left\langle \mathscr{P}^{-1} \left( f \right)_{u, t}, \tau \right\rangle .
	\end{equation}

The definition of $B$ implies that for all $\sigma$, there exists a polynomial $P = P_{\sigma}$ such that $\sigma = P_{\sigma} ( B_{\leq | \sigma |} )$.
As a consequence, note that:
	\begin{equation}
		\left\langle \mathscr{P}^{-1} \left( f \right)_{s, u} , \sigma \right\rangle = \left| s - u \right|^{| \sigma | \alpha} P_{\sigma} \left( \left( \frac{\left\langle \mathscr{P}^{-1} \left( f \right)_{s, u} , b \right\rangle}{\left| s - u \right|^{| b | \alpha}} \right)_{b \in B_{\leq | \sigma |}} \right) .
	\end{equation}

This yields:	
	\begin{equation}
		\left\| \delta \left\langle \mathscr{P}^{-1} \left( f \right), h \right\rangle \right\|_{\mathcal{C}_3^{|h| \alpha}} \leq \sum\limits_{\substack{\sigma, \tau}} P_{\sigma} P_{\tau} \left( \left( \left\| \left\langle \mathscr{P}^{-1} \left( f \right), b \right\rangle \right\|_{\mathcal{C}_2^{|b| \alpha}} \right)_{b \in B_{\leq n}} \right) ,
	\end{equation}
 which is bounded by the recurrence hypothesis.
We conclude the recurrence step using \eqref{eq:recursive_construction_for_the_inverse_of_P} and the continuity of the Sewing map $\Lambda$, so that $\mathscr{P}^{-1}$ does indeed map bounded sets to bounded sets.

Now let us tackle the continuity of $\mathscr{P}^{-1}$.
We reason as above: we fix $f, g \in \mathscr{C}_B^{\alpha}$ and set
$\zeta:=\mathscr{P}^{-1} \left( f \right) - \mathscr{P}^{-1} \left( g \right)$.
We estimate recursively $\left\| \left\langle \zeta, h \right\rangle \right\|_{\mathcal{C}_2^{| h | \alpha}}$ for monomials $h \in B_{\leq N}$.
When $h \in B_1$, by construction of $\mathscr{P}^{-1}$, it holds that $\left\langle \zeta, h \right\rangle = f^h - g^h$, thus we have $$\left\| \left\langle \zeta, h \right\rangle \right\|_{\mathcal{C}_2^{| h | \alpha}} \leq d_{\mathscr{C}_B^{\alpha}} \left( f, g \right).$$
Fix $n \geq 1$, and $h \in B_{n + 1}$.
By the construction above, $\left\langle \zeta , h \right\rangle = \Lambda \left( \delta \left\langle \zeta, h \right\rangle \right) + f^h - g^h$.
Using Chen's relation and the notations above, one obtains a decomposition of the form, denoting $X \coloneqq \mathscr{P}^{-1} \left( f \right)$, $Y \coloneqq \mathscr{P}^{-1} \left( g \right)$:
	\begin{align}
		& \left| \frac{\delta \left\langle X - Y , h \right\rangle_{s, u, t}}{\left( \left| t - u \right| + \left| u - s \right| \right)^{\alpha}} \right| \leq \sum\limits_{\substack{\sigma, \tau}} \Bigg| P_{\sigma} \left( \left( \frac{\left\langle X_{s, u} , b \right\rangle}{\left| s - u \right|^{| b | \alpha}} \right)_{b} \right) P_{\tau} \left( \left( \frac{\left\langle X_{u, t} , b \right\rangle}{\left| t - u \right|^{| b | \alpha}} \right)_{b} \right) -\\
		& \quad \quad \quad \quad \quad \quad \quad \quad \quad \quad \quad - P_{\sigma} \left( \left( \frac{\left\langle Y_{s, u} , b \right\rangle}{\left| s - u \right|^{| b | \alpha}} \right)_{b} \right) P_{\tau} \left( \left( \frac{\left\langle Y_{u, t} , b \right\rangle}{\left| t - u \right|^{| b | \alpha}} \right)_{b} \right) \Bigg| .
	\end{align}
	

Using the fact that polynomials are locally Lipschitz, the fact established above that $\mathscr{P}^{-1}$ maps bounded sets to bounded sets, and the continuity of $\Lambda$, one propagates the locally Lipschitz bound over $h \in B_{n+1}$, so that the theorem is proved by recurrence.
\end{proof}

\section{Sewing versus Reconstruction}
\label{sec:link}

In this section, we discuss the link between the Sewing Lemma (\autoref{thm:usual_Sewing_lemma}, \autoref{thm:simple_formulation_of_weak_Sewing}) and the Reconstruction Theorem \cite[Theorem 3.10]{MR3274562}-\cite[Theorem~5.1]{caravenna2020hairers} in Regularity Structures. It is often claimed that the latter is a generalisation of the former. 
Here we want to test this claim in details, and show that it is correct only up to a point. We are going to see that the Sewing Lemma is
actually slightly stronger than the 1-dimensional version of the Reconstruction Theorem, see \autoref{weaker} below.

The take-home message is the following: if one wants to prove the Sewing Lemma (for any $\gamma>0$) via the Reconstruction Theorem,
then the \enquote{coherence} condition $\delta A\in\mathcal{C}_3^{\gamma}( \Delta_T^3 )$ is not enough, and one needs the \enquote{homogeneity} condition $A\in\mathcal{C}_2^{\beta}( \Delta_T^2 )$ for some $\beta>0$. The original Sewing Lemma on the
other hand holds with the coherence condition only. 

The reason for that is the following: given $A_{s,t}$, one defines the distribution $F_s:=\partial_t A_{s,\cdot}$ depending on the parameter $s$; the Reconstruction Theorem gives the existence of a distribution $f$ with a desired property; in order to obtain the integral $I=\cI(A)$,
one need now to find a "primitive" of $f$, and this is the point where one need the homogeneity condition $A\in\mathcal{C}_2^{\beta}( \Delta_T^2 )$. In other words, by \autoref{prop:Ih}, the Reconstruction Theorem yields the Sewing Lemma only in the case where the primitive
$I=\cI(A)$ belongs to a H\"older space $\mathcal{C}^\beta_1$.
The Sewing Lemma on the other hand does not need to differentiate $A$ and the final integration step with the
associated {homogeneity} condition is therefore unnecessary, see \autoref{weaker}.

\subsection{Reconstruction}
The Reconstruction Theorem is a result in analysis which was first established in the context of the theory of regularity structures \cite{MR3274562}.
It has later been revisited in \cite{caravenna2020hairers}, where it was stated and proved in a more elementary fashion, using only the language of distribution theory.

The Reconstruction Theorem in H\"older spaces is stated in the following way, see \cite{caravenna2020hairers} for a proof and a discussion of this result. For $r \in \mathbb{N}$ we define $\mathscr{B}^r \coloneqq \left\lbrace \varphi \in \mathcal{D} \left( B \left( 0, 1 \right) \right), \left\| \varphi \right\|_{\mathcal{C}^r} \leq 1 \right\rbrace$. For $\varphi \in \mathcal{D} ( \mathbb{R}^d )$, $x \in \mathbb{R}^d$, $\lambda > 0$, we denote $\varphi_x^{\lambda}$ the scaled and recentered version of $\varphi$, defined as follows: $\varphi_x^{\lambda} \left( \cdot \right) \coloneqq \lambda^{- d} \varphi \left( \lambda^{-1} \left( \cdot - x \right) \right)$.

\begin{theorem}[Reconstruction Theorem {\cite[Theorem~5.1]{caravenna2020hairers}}]\label{thm:reconstruction_theorem}
Let $\left( F_x \right)_{x \in \mathbb{R}^d}$ be a germ in the sense of \cite{caravenna2020hairers} i.e.\ for all $x$, $F_x \in \mathcal{D}^{\prime} ( \mathbb{R}^d )$ and for all $\varphi \in \mathcal{D} ( \mathbb{R}^d )$, $x \mapsto F_x \left( \varphi \right)$ is measurable.

Let $a, c \in \mathbb{R}$ with $a \leq 0 \wedge c$.
Assume there exists a test-function $\varphi \in \mathcal{D} ( \mathbb{R}^d )$ with $\int \varphi \neq 0$ such that for all compact $K \subset \mathbb{R}^d$:
	\begin{align}
		\left| \left( F_{x + h} - F_x \right) \left( \varphi_x^{\lambda} \right) \right| & \lesssim \lambda^{a} \left( \left| h \right| + \lambda \right)^{c - a} \label{eq:reconstruction_condition_coherence} , 
	\end{align}
uniformly over $x \in K$, $h \in B \left( 0, 1 \right)$, $\lambda \in \left( 0, 1 \right]$.	
Then there exists a distribution $\mathcal{R} (F) \in \mathcal{D} ( \mathbb{R}^d )$ such that for all compact $K \subset \mathbb{R}^d$ there exists an integer $r_K$ such that:
	\begin{equation}\label{eq:reconstruction_bound}
		\left| \left( R (F) - F_x \right) \left( \psi_x^{\lambda} \right) \right| \lesssim
			\begin{dcases}
		 		\lambda^{c} & \text{ if } c \neq 0 , \\
		 		1 + \left| \log | \lambda | \right| & \text{ if } c = 0 ,
		 	\end{dcases}
	\end{equation}
uniformly over $x \in K$, $\lambda \in \left( 0, 1 \right]$, $\psi \in \mathscr{B}^{r_K}$.

Furthermore, assume that there exists $b < 0 \wedge c$ and a test-function $\varphi \in \mathcal{D} ( \mathbb{R}^d )$ with $\int \varphi \neq 0$ such that for all compact $K \subset \mathbb{R}^d$:
	\begin{equation}
		\left| F_x \left( \varphi_x^{\lambda} \right) \right| \lesssim \lambda^{b} \label{eq:reconstruction_condition_homogeneity} ,
	\end{equation}	 
\noindent uniformly over $x \in K$, $\lambda \in \left( 0, 1 \right]$.
Then one can take $r_K = \max \left( - a, - b \right)$ in \eqref{eq:reconstruction_bound}, and it holds that $\mathcal{R} (F) \in \mathcal{C}_1^{b}$, in the sense that for all integer $s > - b$ and $K \subset \mathbb{R}^d$:
	\begin{equation}
		\left| R (F) \left( \psi_x^{\lambda} \right) \right| \lesssim \lambda^{b} ,
	\end{equation}
\noindent uniformly over $x \in K$, $\lambda \in \left( 0, 1 \right]$, $\psi \in \mathscr{B}^{s}$.
\end{theorem}

Let us briefly comment on the relevance of this result.
The Reconstruction Theorem states that one can retrieve a distribution $\mathcal{R} (F)$ from a collection $\left( F_x \right)_{x \in \mathbb{R}^d}$ of \enquote{local approximations}, under a suitable assumptions of \enquote{coherence} \eqref{eq:reconstruction_condition_coherence} on $F$.
If we furthermore assume a condition of \enquote{homogeneity} \eqref{eq:reconstruction_condition_homogeneity} on $F$, then we have information on the regularity of $\mathcal{R} ( F )$ as a distribution.

The similarity with the Sewing Lemma is evident, and we will indeed show that these two results are intimately related to each other.
Let $0 < \gamma < 1$ and $A \colon [ 0, T ] \to \mathbb{R}$ be a continuous function satisfying
	\begin{equation}
		\left| A_{s, t} - A_{s, u} - A_{u, t} \right| \lesssim \left( | t - u | \vee | u - s | \right)^{\gamma} ,
	\end{equation}
\noindent uniformly over $s, u, t \in \left[ 0, T \right]$.
Let us try to construct a function $I$ satisfying the sewing bound
	\begin{equation} \label{eq:sewing_bound_for_unordered_times}
		| I_t - I_s - A_{s, t} | \lesssim | t - s |^{\gamma} ,
	\end{equation}
\noindent uniformly over $s, t \in \left[ 0, T \right]$, by invoking the Reconstruction Theorem. 
We first define a germ $F$ by differentiating $A$ in the second variable.
More precisely, we extend $A$ to $\mathbb{R}^2$ by setting for $s, t \in \mathbb{R}^2$, $A_{s, t} \coloneqq A_{p \left( s \right), p \left( t \right)}$, where $p \colon s \mapsto \max \left( 0, \min \left( s, T \right) \right)$; then we consider the germ defined for $s \in \mathbb{R}$ by $F_s \coloneqq \partial_t A_{s, \cdot}$, where the partial derivative is understood in the sense of distributions.
That is, for test-functions $\varphi \in \mathcal{D} \left( \mathbb{R} \right)$:
	\begin{equation}
		F_{s} \left( \varphi \right) = - \int_{\mathbb{R}} A_{s, t} \, \varphi^{\prime} (t) \d t .
	\end{equation}

Observe that $F$ satisfies the coherence condition \eqref{eq:reconstruction_condition_coherence} with parameters $a = -1, c = \gamma - 1$ because if $\varphi \in \mathcal{D} \left( \mathbb{R} \right)$ is \emph{any} test-function, $s, u, t \in \mathbb{R}$, $\lambda \in \left( 0, 1 \right]$, then (note that $p$ is $1$-Lipschitz):
			\begin{equation}\begin{split}
				\left| \left( F_t - F_s \right) \left( \varphi_s^{\lambda} \right) \right| & = \lambda^{-1} \left| \int \left( \delta A_{p \left( t \right), p \left( s \right), p \left( s + \lambda v \right)} \right) \varphi^{\prime} (v) \d v \right| 
				\\ & \lesssim \lambda^{-1} \left( \lambda+ \left| t - s \right| \right)^{\gamma} .
			\end{split}\end{equation}

Thus, we can apply \autoref{thm:reconstruction_theorem}, so that there exists $\mathcal{R} (F) \in \mathcal{D}^{\prime} ( \mathbb{R} )$ and integers $r_K$, such that for all compact $K \subset \mathbb{R}$:
	\begin{equation}\label{eq:reconstruction_bound_for_specific_germ}
		\sup\limits_{x \in K} \sup\limits_{\lambda \in \left( 0, 1 \right]} \sup\limits_{\psi \in \mathscr{B}^{r_K}} \frac{\left| \left( F_s - \mathcal{R} (F) \right) \left( \psi_s^{\lambda} \right) \right|}{\lambda^{\gamma - 1}} < + \infty .
	\end{equation}
	
Now in order to obtain the sewing bound \eqref{eq:sewing_bound_for_unordered_times} on $A$, we want to \enquote{integrate} the reconstruction bound \eqref{eq:reconstruction_bound_for_specific_germ} just above.

Heuristically, this \enquote{integration} can be performed by testing the reconstruction bound against indicator functions: notice that $( \mathds{1}_{[0,1]} )_s^{\lambda} = \lambda^{-1} \mathds{1}_{\left[ s, s + \lambda \right]}$, and if $I$ denotes a primitive of $\mathcal{R} ( F )$, one should have $\left( \mathcal{R} (F) - F_s \right) ( \mathds{1}_{\left[ s, s + \lambda \right]} ) = - \left( I - A \right)_{s, s + \lambda}$.
Taking $\psi \coloneqq \mathds{1}_{\left( 0, 1 \right)}$ in the reconstruction bound above would then yield $| \left( I - A \right)_{s, s + \lambda} | \lesssim \lambda^{\gamma}$, which is the expected sewing bound because $\left( I - A \right)_{s, t} = I_t - I_s - A_{s, t}$.

However, this heuristic argument fails for two reasons:
	\begin{enumerate}
		\item $\mathds{1}_{\left( 0, 1 \right)}$ is not a test-function (because it is not smooth) and thus cannot be plugged into the reconstruction bound,
		\item without further assumptions it is not clear in general whether $\mathcal{R} ( F )$ admits a primitive $I$ which is a function.
\end{enumerate}	

\subsection{Approximating an indicator function} 
We will solve the first point above by suitably approximating $\mathds{1}_{\left( 0, 1 \right)}$.
Specifically, we shall exploit the following decomposition, see \cite[Exercise~13.10]{MR4174393} for a similar statement. 
\begin{Lemma}[Dyadic approximation of indicator functions] \label{lemma:dyadic_decomposition}
There exist smooth functions $\varphi_n, \psi_n$, $n \in \mathbb{N}$ such that:
	\begin{enumerate}
		\item for all $n \in \mathbb{N}$, $\supp \left( \varphi_n \right) \subset \left[ \frac{1}{16} 2^{-n}, \frac{15}{16} 2^{-n} \right]$,
		\item for all $n \in \mathbb{N}$, $\supp \left( \psi_n \right) \subset \left[ 1 - \frac{15}{16} 2^{-n}, 1 - \frac{1}{16} 2^{- n} \right]$,
		\item for all $r \in \mathbb{N}$, $\sup\limits_{n \in \mathbb{N}} \sup\limits_{k \in \left\llbracket 0, r \right\rrbracket} \frac{\left\| D^k \varphi_n \right\|_{\infty} + \left\| D^k \psi_n \right\|_{\infty}}{2^{kn}} < + \infty$, 
		\item $\mathds{1}_{\left( 0, 1 \right)} = \sum_{n \geq 0} \left( \varphi_n + \psi_n \right)$.
	\end{enumerate}
\end{Lemma}

\begin{proof}
It is well known \cite[Proposition~2.10]{MR2768550} that there exists a function $\varphi \in \mathcal{C}_c^{\infty} ( [ 1/2, 2 ] )$ such that for all 
$x \in \mathbb{R}$, $\mathds{1}_{\mathbb{R}_+^{*}} \left( x \right) = \sum_{n \in \mathbb{Z}} \varphi \left( 2^n x \right)$.
For $x \in \mathbb{R}$, set $$\eta \left( x \right) \coloneqq \mathds{1}_{\left( 0, 1 \right)} \left( x \right) - \sum_{n \geq 3} \left( \varphi \left( 2^n x \right) + \varphi \left( 2^n \left( 1 - x \right) \right) \right).$$
Observe that $\eta$ is a test-function supported in $[ 1/{16} , 15/{16} ]$, whence the announced decomposition.
\end{proof}

As a consequence:
\begin{prop}[Approximation argument]
In the setting above, assume that there exists a continuous function $I$ such that $I_0 = 0$ and $I^{\prime} = \mathcal{R} (F)$ in the sense of distributions.
Set, for $n \in \mathbb{N}$, $s \in [0,T]$, $\lambda > 0$:
	\begin{equation}
		\Delta_{s, \lambda}^{N} \coloneqq \sum\limits_{n = 0}^N \left( \mathcal{R} (F) - F_s \right) \left( \lambda \left( \varphi_n + \psi_n \right)_s^{\lambda} \right) .
	\end{equation}
Then:
	\begin{enumerate}
		\item for all $s \in \mathbb{R}$, $\lambda > 0$, $\lim_{N \to \infty} \Delta_{s, \lambda}^N = \left( I - A \right)_{s, s + \lambda}$ ,
		\item For all compact $K \subset \mathbb{R}$, $| \Delta_{s, \lambda}^{N} | \lesssim \lambda^{\gamma}$ uniformly over $s \in K$, $N \in \mathbb{N}$, $\lambda \in \left( 0, 1 \right]$.
	\end{enumerate}		
\end{prop}
\begin{proof}
Let us establish those points separately.
On the one hand, since for $s \in \mathbb{R}$, $\mathcal{R} (F) - F_s = \left( I_{\cdot} - A_{s, \cdot} \right)^{\prime}$, we have for $N \in \mathbb{N}$:
			\begin{align}
				\Delta_{s, \lambda}^N & = - \int \left( I_u - A_{s, u} \right) \sum\limits_{n = 0}^N \left( \left( \varphi_n + \psi_n \right)^{\prime} \right)_s^{\lambda} \left( u \right) du .
			\end{align}
		 By construction of $\varphi_n$ and $\psi_n$, for all $x \in [ 2^{- N}, 1 - 2^{-N} ]$ one has $\sum_{n = 0}^N \left( \varphi_n \left( x \right) + \psi_n \left( x \right) \right) = 1$.
		 Hence, setting for $N \in \mathbb{N}$ and $u \in \mathbb{R}$:
			\begin{equation}
				\begin{dcases}
					\eta_N \left( u \right) & \coloneqq 2^{- N} \sum\limits_{n = 0}^N \left( \varphi_n + \psi_n \right)^{\prime} \left( 2^{-N} u \right) \mathds{1}_{[0,T]} \left( u \right) , \\
					\tilde{\eta}_N (v) & \coloneqq - 2^{- N} \sum\limits_{n = 0}^N \left( \varphi_n + \psi_n \right)^{\prime} \left( 1 + 2^{-N} u \right) \mathds{1}_{\left[ -1, 0 \right]} \left( u \right) ,
				\end{dcases}
			\end{equation}
		
		\noindent then it follows that $\eta_N, \tilde{\eta}_N \in \mathcal{D} \left( B \left( 0, 1 \right) \right)$, and $\sum_{n = 0}^N \left( \varphi_n + \psi_n \right)^{\prime} = \left( \eta_N \right)_{0}^{2^{-N}} - \left( \tilde{\eta}_N \right)_{1}^{2^{-N}}$.
		Note that $\left\| \tilde{\eta}_N \right\|_{L^{\infty}}, \left\| \eta_N \right\|_{L^{\infty}} \lesssim 1$ and $\int_{\mathbb{R}} \eta_N \left( u \right) d u = \int_{\mathbb{R}} \tilde{\eta}_N \left( u \right) d u = 1$.
		Hence:
			\begin{align}
				\Delta_{s, \lambda}^N & = - \int_{} I_{s + \lambda 2^{- N} u} \eta_N \left( u \right) d u
				 + \int_{} A_{s, s + \lambda 2^{- N} u} \eta_N \left( u \right) d u \\
				& \quad + \int_{} I_{s + \lambda + \lambda 2^{- N} u} \tilde{\eta}_N \left( u \right) d u 
				 - \int_{} A_{s, s + \lambda + \lambda 2^{- N} u} \tilde{\eta}_N \left( u \right) d u.
			\end{align}
			
		Let us treat the first term.
		We write:
			\begin{equation} \label{eq:mollification_with_dominated_convergence}
				\int_{} I_{s + \lambda 2^{- N} u} \eta_N \left( u \right) d u = I_{s} + \int_{} \left( I_{s + \lambda 2^{- N} u} - I_s \right) \eta_N \left( u \right) d u . 
			\end{equation}
			
		Because $I$ is continuous, the integrand in the right-hand side converges pointwise to $0$ and is bounded by a constant.
		By the dominated convergence theorem, it follows that the sequence of integrals in \eqref{eq:mollification_with_dominated_convergence} converges to $0$.
		Reasoning similarly for the other terms, one obtains as announced that $\Delta_{s, \lambda}^N = I_{s + \lambda} - I_s - A_{s, s + \lambda} + o_{N \to + \infty} \left( 1 \right)$.		
	
Now let us bound $| \Delta_{s, \lambda}^N |$.
For $n \in \mathbb{N}$ and $x \in \mathbb{R}$, set:
			\begin{equation}
				\eta_n \left( x \right) \coloneqq \varphi_n \left( 2^{-n} x \right) , \quad 
				\tilde{\eta}_n \left( x \right) \coloneqq \psi_n \left( 2^{-n} x + 1 \right) .
			\end{equation}
		Then $\left( \varphi_n \right)_s^{\lambda} = 2^{-n} \left( \eta_n \right)_{s}^{2^{-n} \lambda}$, and $\left( \psi_n \right)_s^{\lambda} = 2^{-n} \left( \tilde{\eta}_n \right)_{s + \lambda}^{2^{-n} \lambda}$.
		Also, note that the quantity $C_K \coloneqq \sup_{n \in \mathbb{N}} \left( \left\| \eta_n \right\|_{\mathcal{C}^{r_K}} + \left\| \tilde{\eta}_n \right\|_{\mathcal{C}^{r_K}} \right)$ is finite, so that $\eta_n / C_K, \tilde{\eta}_n / C_K \in \mathscr{B}^2$.
		Then we have the decomposition:
			\begin{align}
				\Delta_{s, \lambda}^N & = \underbrace{\sum\limits_{n = 0}^N \left( \mathcal{R} (F) - F_s \right) \left( \lambda 2^{-n} \left( \eta_n \right)_s^{\lambda 2^{- n}} \right)}_{\eqqcolon \Delta_{s, \lambda}^{N; 1}} \\ & \quad + \underbrace{\sum\limits_{n = 0}^N \left( \mathcal{R} (F) - F_{s + \lambda} \right) \left( \lambda 2^{- n} \left( \tilde{\eta}_n \right)_{s + \lambda}^{\lambda 2^{-n}} \right)}_{\eqqcolon \Delta_{s, \lambda}^{N; 2}} \\
				& \quad + \underbrace{\sum\limits_{n = 0}^N \left( F_{s + \lambda} - F_s \right) \left( \lambda 2^{-n} \left( \tilde{\eta}_n \right)_{s + \lambda}^{\lambda 2^{- n}} \right)}_{\eqqcolon \Delta_{s, \lambda}^{N; 3}} .
			\end{align}
		Using the reconstruction bound \eqref{eq:reconstruction_bound_for_specific_germ} on the $1$-enlargement of $K$, one has $\Delta_{s, \lambda}^{N; 1} \lesssim \sum_{n = 0}^N \lambda 2^{-n} \left( \lambda 2^{- n} \right)^{\gamma - 1} \lesssim \lambda^{\gamma}$.
		Similarly, $\Delta_{s, \lambda}^{N; 2} \lesssim \lambda^{\gamma}$, so that it only remains to treat $\Delta_{s, \lambda}^{N; 3}$.
		For this purpose, we rewrite:
			\begin{equation}
				\Delta_{s, \lambda}^{N; 3} =  - \lambda \int_{\mathbb{R}} \delta A_{s + \lambda, s, u} \sum\limits_{n = 0}^N 2^{- n} \left( \left( \tilde{\eta}_n \right)_{s + \lambda}^{\lambda 2^{- n}} \right)^{\prime} \left( u \right) d u .
			\end{equation}
		Note that $\supp \left( \left( \tilde{\eta}_n \right)_{s + \lambda}^{\lambda 2^{- n}} \right) \subset \left[ s + \lambda, s + \lambda + 2^{- n} \lambda \right]$, and thus:
			\begin{equation}
				\left| \Delta_{s, \lambda}^{N; 3} \right| \leq \left( \sup\limits_{u \in \left[ s + \lambda, s + 2 \lambda \right]} \left| \delta A_{s + \lambda, s, u} \right| \right) \int_{\mathbb{R}} \left| \sum\limits_{n = 0}^N \lambda 2^{- n} \left( \left( \tilde{\eta}_n \right)_{s + \lambda}^{\lambda 2^{- n}} \right)^{\prime} \left( u \right) \right| d u .
			\end{equation}	
		Now, the hypothesis on $A$ implies that for $u \in \left[ s + \lambda, s + 2 \lambda \right]$, we have $\left| \delta A_{s + \lambda, s, u} \right| \lesssim \left( \left| u - s \right| + \left| s - \left( s + \lambda \right) \right| \right)^{\gamma} \lesssim \lambda^{\gamma}$.
		Now it remains to establish that the sequence of integrals above is bounded.
		We write:
			\begin{align}
				\int_{\mathbb{R}} \left| \sum\limits_{n = 0}^N \lambda 2^{- n} \left( \left( \tilde{\eta}_n \right)_{s + \lambda}^{\lambda 2^{- n}} \right)^{\prime} \left( u \right) \right| d u = \int_{\mathbb{R}} \left| f_{N}^{\prime} (v) \right| d v .
			\end{align}
where we set $f_N \colon v \mapsto \sum_{n = 0}^N \psi_n \left( 1 - v \right)$.
		Note by definition of $\psi$ that $\supp \left( f_N \right) \subset [0,T]$, and that $f_N$ converges pointwise to $f \coloneqq \sum_{n = 0}^{+ \infty} \psi_n \left( 1 - v \right)$.
		From the definition of $\varphi_n, \psi_n$, one has $f = \mathds{1}_{\left( 0, 1 \right)} - \sum_{n = 0}^{+ \infty} \varphi_n \left( 1 - \cdot \right)$ so that from the properties of the supports of $\varphi_n, \psi_n$, one has $f \mathds{1}_{\left( 0, \frac{1}{16} \right)} = \mathds{1}_{\left( 0, \frac{1}{16} \right)}$, and thus $\left\| f^{\prime} \right\|_{L^{\infty} \left( 0, 1 \right)} < + \infty$.			
		Also, by construction, $\supp \left( \psi_n \left( 1 - \cdot \right) \right) \subset \left[ \frac{1}{16} 2^{- n}, \frac{15}{16} 2^{-n} \right]$.	
		Thus:
			\begin{align}
				\int_{\mathbb{R}} \left| f_N^{\prime} (v) \right| d v & = \int_{0}^{2^{-N}} \sum\limits_{n = N - 4}^N \left| \psi_n^{\prime} \left( 1 - v \right) \right| d v + \int_{2^{-N}}^{1} \left| f^{\prime} (v) \right| d v \\
				& \leq 5 \sup\limits_{n \in \mathbb{N}} \left( 2^{- n} \left\| \psi_n \right\|_{\mathcal{C}^1} \right) + \left\| f^{\prime} \right\|_{L^{\infty} \left( 0, 1 \right)} \lesssim 1 .
			\end{align}
		Thus we have established that $\left| \Delta_{s, \lambda}^{N} \right| \lesssim \lambda^{\gamma}$.		
This concludes the proof.
\end{proof}

Recall that the calculations above assume the existence of a continuous primitive $I$ of $\mathcal{R} (F)$, which is far from clear in general.
However, assume now that there exists $0 < \beta < 1 \wedge \gamma$ such that $A$ furthermore satisfies:
	\begin{equation}
		\left| A_{s, t} \right| \lesssim \left| t - s \right|^{\beta} ,
	\end{equation}

\noindent uniformly over $s, t \in \left[ 0, T \right]$.
Then, note that $F$ satisfies the homogeneity condition \eqref{eq:sewing_condition_homogeneity} with parameter $b = \beta - 1$, because if $\varphi \in \mathcal{D} \left( \mathbb{R} \right)$ is any test-function, $s, t \in \mathbb{R}$, $\lambda \in \left( 0, 1 \right]$, then:
			\begin{align}
				\left| F_s \left( \varphi_s^{\lambda} \right) \right| &= \lambda^{-1} \left|  A_{p \left( s \right), p \left( s + \lambda v \right)} \varphi^{\prime} (v) dv \right| \lesssim \lambda^{\beta - 1}.
			\end{align}

Thus, the Reconstruction \autoref{thm:reconstruction_theorem} asserts that $\mathcal{R} ( F ) \in \mathcal{C}_1^{\beta - 1}$, and it is then a well-known fact that there exists a function $I \in \mathcal{C}_1^{\beta}$ such that $I_0 = 0$ and $I^{\prime} = \mathcal{R} ( F )$.
\begin{remark}
A rigorous proof of the latter fact can be found for example in \cite[Lemma~3.10]{MR4291780} where the author uses wavelets, but let us briefly and informally present an alternative approach in the spirit of the calculations above.
Indeed, one wishes to set $I_t \coloneqq \mathcal{R} \left( F \right) ( \mathds{1}_{( 0 \wedge t, 0 \vee t]} )$, which is not possible because the indicator function is not a test-function.
However, recalling \autoref{lemma:dyadic_decomposition}, one can define for $t \geq 0$:
	\begin{equation}
		I_t \coloneqq \sum\limits_{n \geq 0} \mathcal{R} \left( F \right) \left( t \left( (\varphi_n)^{(t)} + (\psi_n)^{(t)} \right) \right) ,
	\end{equation}

\noindent where the sum is absolutely convergent because the fact that $\mathcal{R} ( F ) \in \mathcal{C}_1^{\beta - 1}$ and the properties of $\varphi_n$ and $\psi_n$ imply that the absolute value of the summand is bounded by a constant times $2^{- n \beta}$, which is summable because $\beta > 0$.
Now one can verify that indeed $I \in \mathcal{C}_1^{\beta}$ and that $I^{\prime} = \mathcal{R} \left( F \right)$.
\end{remark}

Hence, one retrieves the following weaker version of the Sewing Lemma (\autoref{thm:usual_Sewing_lemma}, \autoref{thm:simple_formulation_of_weak_Sewing}) above. Note that the same arguments can be adapted to the case $\gamma = 1$.
\begin{corollary}[Sewing via the Reconstruction {\autoref{thm:reconstruction_theorem}}]\label{thm:weak_sewing_with_homogeneity}
Let $\beta, \gamma > 0$ with $\beta < 1$, and let $A \colon [0,T]^2 \to \mathbb{R}$ be a function satisfying:
		\begin{align}
				\left| A_{s, t} - A_{s, u} - A_{u, t} \right| & \lesssim \left( \left| t - u \right| \vee \left| u - s \right| \right)^{\gamma} \label{eq:sewing_condition_coherence} , \\
				\left| A_{s, t} \right|& \lesssim \left| t - s \right|^{\beta} \label{eq:sewing_condition_homogeneity} ,
		\end{align}

\noindent uniformly over $s, u, t \in [0,T]$.
Then there exists a function $I \in \mathcal{C}_1^{\beta}$ such that:
	\begin{equation}
		| I_t - I_s - A_{s, t} | \lesssim 
			\begin{dcases}
				\left| t - s \right|^{\gamma} & \text{ if } \gamma \neq 1 , \\
				|t-s|\left( 1+\left| \log |t-s| \right| \right)  & \text{ if } \gamma = 1 ,
			\end{dcases}
	\end{equation}
\noindent uniformly over $s, t \in [0,T]$
\end{corollary}

\begin{remark}\label{weaker}
Note that \autoref{thm:weak_sewing_with_homogeneity} is weaker than \autoref{thm:usual_Sewing_lemma}-\autoref{thm:simple_formulation_of_weak_Sewing} above because it further requires to assume the condition of homogeneity \eqref{eq:sewing_condition_homogeneity}.
Indeed our proof requires the assumption \eqref{eq:sewing_condition_homogeneity} above in order to retrieve this result as an application of the Reconstruction Theorem.
Still, it is not clear to us whether the Reconstruction Theorem can yield the more general \autoref{thm:usual_Sewing_lemma}, \autoref{thm:simple_formulation_of_weak_Sewing}.
\end{remark}

\printbibliography

\end{document}